\documentclass[preprint,12pt]{elsarticle}

\usepackage{amssymb}
\usepackage{amsmath}
\usepackage{amsthm}
\usepackage{mathtools}
\usepackage{multicol}
\usepackage{algorithm, algpseudocode}
\usepackage{multirow}
\usepackage{eurosym}

\newtheorem{proposition}{Proposition}
\newtheorem{lemma}{Lemma}

\allowdisplaybreaks

\journal{European Journal of Operational Research}

\begin{document}

\begin{frontmatter}



\title{What Are We Clustering For? Establishing Performance Guarantees for Time Series Aggregation in Generation Expansion Planning}

\author[label1]{Luca Santosuosso}
\author[label2]{Bettina Klinz}
\author[label1]{Sonja Wogrin}
            
\affiliation[label1]{
            organization={Institute of Electricity Economics and Energy Innovation, Graz University of Technology},
            addressline={Inffeldgasse 18},
            city={Graz},
            postcode={8010},
            country={Austria}}

\affiliation[label2]{
            organization={Institute for Discrete Mathematics, Graz University of Technology},
            addressline={Steyrergasse 30},
            city={Graz},
            postcode={8010},
            country={Austria}}

\begin{abstract}
Generation expansion planning (GEP) is a prominent example of capacity expansion problems in operations research.
Being generally NP-hard, GEP optimization models can become intractable when nonconvex dynamics, time-coupling constraints, and complex asset interactions are involved.
Time series aggregation (TSA) tackles this by reducing temporal complexity via input data clustering.
However, existing TSA methods either focus solely on preserving the statistical features of the input data,
yielding heuristics without guarantees on the aggregated model's accuracy,
or provide error bounds limited to linear models, neglecting time-coupling constraints and applying only to specific clustering techniques.
Moreover, these bounds typically pertain solely to the GEP objective function and do not extend to other stakeholder-specific metrics, such as decision vector partitions.
To tackle these issues, we demonstrate that an appropriately constructed aggregated model always provides a lower bound on the optimal objective function value of the full-scale GEP model in both mixed-integer linear and mixed-integer quadratic formulations with time-coupling, independent of the clustering technique employed.
Building on this, we propose a performance-guaranteed TSA-based solution algorithm that iteratively refines objective function bounds while generating feasible solutions to the full-scale model at each iteration.
We then discuss a comparison with Benders decomposition and demonstrate how the derived bounds can be extended to error estimates for stakeholder-specific metrics.
Numerical results show the computational advantages of our method over both full-scale optimization and classical Benders decomposition.
\end{abstract}




\end{frontmatter}


\section{Introduction}
\label{sec:Introduction}
Generation expansion planning (GEP) is a foundational problem in power system optimization \cite{conejo2016investment} and ranks among the most thoroughly studied capacity expansion problems in operations research \cite{luss1982operations}.
At its core, GEP typically involves determining the optimal mix, sizing, and timing of investments in power generation units to meet future electricity demand at minimal cost, while satisfying operational, economic, and environmental constraints \cite{koltsaklis2018state}.

Traditional GEP optimization models were developed in the context of vertically integrated, state-owned utilities operating centralized power systems \cite{palmintier2015impact}.
In this setting, the electricity demand evolved predictably, and the power generation was dominated by large, dispatchable units with limited uncertainty. These conditions enabled simplified representations of the system operational dynamics, making linear programming (LP) formulations of the GEP problem both suitable and computationally efficient \cite{turvey1977electricity}.

Since the 1990s, the power sector has experienced a profound transformation, driven by the liberalization of electricity markets, the growing integration of variable renewable energy sources (vRES), and the widespread deployment of flexibility technologies such as energy storage systems, electric vehicles, and responsive demand \cite{sadeghi2017comprehensive}.
These developments have resulted in increasingly complex power system architectures, marked by strong interdependencies among heterogeneous assets, nonlinear and time-varying operational dynamics, as well as a broader spectrum of stakeholder objectives. 
As a result, modern GEP formulations have advanced into increasingly sophisticated mathematical programs \cite{hobbs1995optimization}, encompassing mixed-integer linear programming (MILP) \cite{kang2023stochastic}, mixed-integer quadratic programming (MIQP) \cite{franco2014mixed}, and mixed-integer nonlinear programming (MINLP) models \cite{gbadamosi2021comparative}, together with hybrid optimization frameworks integrating multiple modeling paradigms \cite{irawan2023integrated}.

While it is widely recognized that the GEP problem has grown increasingly complex \cite{dagoumas2019review}, the mere fact that it is formulated as a MI(N)LP model does not, by itself, provide any indication of its computational complexity. In this regard, the theoretical analysis presented in \cite{goderbauer2019synthesis} formally establishes that the GEP problem is, in general, strongly NP-hard, i.e., there is no polynomial time algorithm for solving it exactly unless $\mathrm{P} = \mathrm{NP}$.
Thus, GEP models, typically involving millions or even billions of variables, are not only computationally demanding but also often computationally intractable \cite{li2022mixed}.

This computational challenge has motivated extensive research into approximation methods \cite{constante2025relaxation}, among which two classes have gained particular prominence: decomposition methods and aggregation methods.
Decomposition methods mitigate computational complexity by partitioning the original GEP problem into smaller, more manageable subproblems, which are then solved iteratively and coordinated to seek convergence toward global optimality.
Among these, applications of the Benders decomposition method to GEP have been widely investigated \cite{liu2024generalized}.
In contrast, aggregation methods maintain the full resolution of the GEP optimization model but reduce its dimensionality by aggregating variables and/or constraints. This results in a surrogate (or aggregated) model that approximates the original full-scale model while significantly lowering its computational complexity.

This study focuses on advancing a specific subset of aggregation methods, namely time series aggregation (TSA),
with particular emphasis on establishing performance guarantees in the form of theoretically validated bounds on the maximum error incurred by the aggregated models.
Specifically, we examine TSA applied to both mixed-integer linear and nonlinear GEP models with time-coupling constraints, addressing the challenge of preserving temporal chronology during aggregation,
while also providing a comparison between the proposed TSA-based solution algorithms and Benders decomposition.

\subsection{Literature Review}
\label{subsec:Literature_Review}
Traditional TSA methods reduce the complexity of optimization models by applying clustering techniques to the input time series \cite{lara2018deterministic}.
Based on the statistical features of the input data, each time period is assigned to a cluster, with each cluster corresponding to a representative time period.
The aggregated optimization model is then formulated over this reduced set of representative periods and their corresponding aggregated input values.
When the number of representative periods is significantly smaller than that of the original time series, solving the aggregated model in place of the full-scale model yields significant computational advantages.

Traditional TSA methods have been extensively employed in GEP \cite{bylling2020impact}, 
and are generally distinguished by the clustering techniques used to identify representative periods.
Common techniques include k-means \cite{jia2020privacy}, k-medoids \cite{sarajpoor2023time}, hierarchical clustering \cite{liu2017hierarchical}, and Gaussian mixture models (GMM) \cite{zhang2020novel}.
These techniques have been thoroughly evaluated in the literature, with a general consensus that their relative performance is highly dependent on the specific characteristics of the problem at hand \cite{kotzur2018impact}.
In response, recent research has increasingly focused on the development of TSA methods that refine clustering techniques to better reflect the characteristics of the input time series, such as load profiles \cite{mets2015two} and vRES power generation \cite{jin2021wind},
as well as to more accurately represent extreme events \cite{blanford2018simulating} and temporal correlations across multiple time series \cite{sarajpoor2021shape}.
Although these tailored approaches can improve the fidelity of aggregated models in specific applications, their performance remains inherently problem-dependent \cite{teichgraeber2022time}.

Moreover, the direct application of conventional clustering techniques for TSA in GEP is often infeasible, particularly in the presence of storage constraints, as the temporal chronology of the model must be preserved \cite{garcia2020representative}.
This issue can be addressed either by modifying the clustering technique to retain temporal chronology \cite{pineda2018chronological}
or by reconstructing it within the aggregated model \cite{tejada2017representation}.
While enforcing chronology during clustering ensures consistency with the full-scale model, it limits algorithmic flexibility and may yield less representative clusters \cite{poncelet2016impact}.
Conversely, allowing unrestricted clustering provides greater flexibility but requires careful design of the aggregated model to restore temporal coherence with its full-scale counterpart \cite{tejada2018enhanced}.

Despite their widespread use, traditional TSA methods rest on the assumption that accurately capturing the statistical features of the input time series yields an accurate aggregated model, a premise that often fails \cite{hoffmann2021typical}.
This limitation has fostered a distinction between traditional \textit{a priori} TSA methods,
which aim to minimize the \textit{input error} by focusing solely on accurately representing the input space through clustering,
and the emerging \textit{a posteriori} TSA methods,
which extend clustering with optimization routines and full-scale model analysis to directly reduce the \textit{output error},
i.e., the discrepancy between aggregated and full-scale model outputs \cite{hoffmann2020review}.

Given that modelers are typically more concerned with accurately representing the model’s output space rather than its input space,
a posteriori TSA has attracted growing interest,
with examples in the literature proposing both data-driven \cite{sun2019data} and optimization-based \cite{poncelet2016selecting} a posteriori methods for GEP.
Notably, \cite{wogrin2023time} demonstrates that, if the modeler were to know the active constraints of an LP problem in advance,
an a posteriori method could be employed to construct an aggregated model that exactly replicates the solution of its full-scale counterpart,
while achieving a dimensionality reduction of three orders of magnitude.
This result highlights the potential of a posteriori TSA to deliver aggregated models with strong \textit{performance guarantees}.

However, the assumption of prior knowledge of the active constraints is generally impractical in real-world problems.
Alternatively, previous research has focused on deriving bounds on the objective function error introduced by aggregated models relative to the full-scale model \cite{huberman1983error}.
Algorithms for computing such bounds have been proposed for GEP problems without \cite{yokoyama2021effect} and with \cite{tso2020hierarchical} storage constraints.
Nonetheless, the validity of these bounds is not formally demonstrated, rendering these algorithms fundamentally heuristic.

The formal derivation of performance guarantees for TSA methods,
particularly in the form of theoretically validated bounds on the maximum error incurred by the aggregated models,
has a long-standing tradition in operations research \cite{alexander2021general}.
Classical studies have primarily focused on LP models,
addressing both row aggregation \cite{zipkin1980bounds} and, in a broader sense, variable aggregation \cite{zipkin1980bounds2}.
More recently, research has shifted toward MILP formulations of the GEP problem.
Building on the Jensen's inequality, \cite{munoz2016new} demonstrated that, in the absence of storage constraints, the optimal objective function value of an appropriately constructed aggregated model provides a valid lower bound on that of the full-scale model.
This theoretical result was extended in \cite{li2022representative} to GEP models with storage constraints, albeit under the assumption that a specific clustering technique, namely, k-means, is used.

Despite these advances, several research gaps remain.
These include deriving clustering-independent objective function bounds via TSA,
the extension of such bounds to MINLP formulations of the GEP with storage constraints, particularly within the increasingly popular MIQP formulations,
the systematic recovery of feasible solutions to the full-scale model from those obtained via the aggregated model,
and the translation of the objective bounds into error bounds on the optimal solution itself.
These challenges continue to drive efforts toward rigorous performance guarantees for TSA methods \cite{fragoso2025solving}.

\subsection{Research Questions and Contributions}
\label{subsec:Motivation_and_Contributions}
The literature review reveals the following open research questions:
\begin{itemize}
    \item \textit{How can performance guarantees be established for aggregated models derived via TSA?}
    Several studies propose TSA-based algorithms to derive bounds on the objective function value error incurred by aggregated models in GEP \cite{yokoyama2021effect, tso2020hierarchical},
    accompanied by formal validations of these bounds \cite{munoz2016new, li2022representative, teichgraeber2019clustering}.
    Nonetheless, these contributions exhibit notable limitations.
    Specifically, the bounds presented in \cite{yokoyama2021effect, tso2020hierarchical} lack formal theoretical validation;
    the theoretical analysis in \cite{teichgraeber2019clustering} is restricted to LP models;
    \cite{munoz2016new} neglects storage constraints, which introduce considerable complexity into the TSA process due to the necessity of preserving temporal chronology; 
    and \cite{li2022representative} provides performance guarantees only when the clustering technique employed for TSA is k-means.
    Notably, none of these studies address MINLP formulations of the GEP problem.
    A TSA method yielding aggregated models with rigorously validated objective function error bounds, independent of the clustering technique employed and preserving the temporal chronology imposed by storage constraints in MI(N)LP GEP models, is currently lacking.
    \item \textit{What are we clustering for?}
    As previously analyzed, existing research on TSA has focused on deriving performance bounds for aggregated models.
    However, such objective function error bounds do not readily extend to the individual cost terms within the objective function or to the model’s optimal solution.
    This limitation is particularly critical in GEP, where similar models are employed by diverse stakeholders, such as energy utilities \cite{lopez2020integrated}, policy makers \cite{pina2011modeling}, and system planners \cite{bruno2016risk}, each focusing on different modeling aspects of the same GEP problem.
    For instance, a wind power investor may rely on a comprehensive GEP model to capture system-wide dynamics, yet their primary concern lies in the investment and operational performance of the wind units alone.
    Deriving error bounds on the aggregated model’s objective function value, which may be dominated by cost terms from other technologies, offers limited relevance.
    This issue is further compounded when using the TSA-based algorithms in \cite{yokoyama2021effect, tso2020hierarchical, munoz2016new, li2022representative, teichgraeber2019clustering}, which are deemed convergent once the optimality gap, i.e., the difference between the derived upper and lower bounds on the objective function value, falls below a predefined threshold:
    although all solutions within this gap satisfy the global bounds,
    only a subset may retain fidelity in the specific modeling aspects relevant to a given stakeholder.
    To date, no study has examined how the objective function bounds derived via TSA methods can be disaggregated into stakeholder-specific performance metrics.
\end{itemize}

This study seeks to bridge these research gaps.
The key contributions of this paper are as follows:
\begin{itemize}
    \item We provide both MILP and MIQP formulations of the GEP problem with storage constraints, together with the corresponding aggregated models derived via TSA.
    We then establish a theoretical result showing that appropriately constructed aggregated models always yield valid lower bounds on the optimal objective function value of their full-scale MILP and MIQP counterparts.
    Notably, this theoretical result holds irrespective of the clustering technique used for TSA.
    \item Building on this theoretical foundation,
    we propose a performance-guaranteed algorithm to solve the GEP problem by leveraging the aggregated model's solution,
    applicable to both MILP and MIQP formulations. The proposed algorithm provides theoretically validated bounds on the objective function error incurred by the aggregated model relative to its full-scale counterpart,
    while delivering a feasible solution to the full-scale model at every iteration. To position the proposed algorithm within the broader landscape of existing approaches for addressing the computational complexity of GEP,
    we compare it with the popular Benders decomposition approach and show its relative effectiveness through numerical results.
    \item Finally, to address the question of \textit{what are we clustering for?},
    we demonstrate how the derived objective function value bounds can be disaggregated into relevant stakeholder-specific performance metrics,
    including individual cost terms within the objective function of the GEP model and selected partitions of the decision vector.
\end{itemize}

In the following, we focus specifically on GEP problems characterized by high penetration of vRES, reflecting the critical role of deeply decarbonized power systems in the ongoing energy transition.
Nevertheless, the proposed theoretical results and solution algorithms are readily applicable to a broad class of capacity expansion problems in domains such as heavy process industries, communication networks, and water resource systems, which exhibit strong modeling parallels with the GEP models considered here \cite{luss1982operations}.

The remainder of the paper is organized as follows:
Section~\ref{sec:Methodology} presents the problem formulation, the theoretical results, and the solution algorithm.
Section~\ref{sec:numerical_results} reports the numerical results, and Section~\ref{sec:conclusion} concludes the study.

\section{Methodology}
\label{sec:Methodology}
This section outlines the proposed methodology.
The GEP problem is introduced in Subsection~\ref{subsec:problem_statement}.
The full-scale models are presented in Subsection~\ref{subsec:full_scale_models}, followed by their aggregated counterparts in Subsection~\ref{subsec:aggregated_models}. 
Subsection~\ref{subsec:properties_aggregated_models} discusses the properties of the aggregated models.
Subsection~\ref{subsec:performance_guaranteed_TSA} introduces the performance-guaranteed TSA-based solution algorithm,
while Subsection~\ref{subsec:comparison_benders_decomposition} compares it with Benders decomposition.
Finally, Subsection~\ref{subsec:what_are_we_clustering_for} extends the performance guarantees to stakeholder-specific metrics.

In the following, vectors and sets are denoted in bold, e.g., $\boldsymbol{x}$, while matrices are represented by bold, non-italicized uppercase symbols, e.g., $\mathbf{A}$.
The symbol $\circ$ denotes the Hadamard product (or element-wise product) between vectors or matrices.
Moreover, $\mathbf{1}_{l\times m}$, $\mathbf{0}_{l\times m}$, and $\mathbf{I}_{m}$ represent the $l \times m$ matrix of ones, the $l \times m$ zero matrix, and the $m \times m$ identity matrix, respectively.

\subsection{Problem Statement}
\label{subsec:problem_statement}
The goal is to determine the optimal mix and sizing of generation units that minimize both capital investment and operational costs, while satisfying the energy demand.
The investment options include thermal generation, vRES (wind and solar), and storage systems. 
Various GEP formulations exist in the literature \cite{munoz2016new}.
Here, we adopt formulations akin to that of \cite{li2022mixed}.

The GEP problem is defined over a discrete set of time periods $\boldsymbol{T}$, indexed by $t$.
It involves a total of $G$ generators and $N$ energy storage systems.
The energy demand in each period is denoted by $D_t$,
while the vector of vRES capacity factors during period $t$ is denoted by $\boldsymbol{F}_{t} \in \mathbb{R}^G$.
The operational costs associated with the generators are captured by the vector $\boldsymbol{C^\mathrm{p}} \in \mathbb{R}^G$,
while the penalty cost for non-supplied energy demand is $C^\mathrm{ns}$.
The capital investment required for expanding the generation capacity is $\boldsymbol{C}^{\boldsymbol{\mathrm{inv}}} \in \mathbb{R}^{G+N}$.

\subsection{The Full-Scale Optimization Models}
\label{subsec:full_scale_models}
This subsection presents two distinct full-scale models for the GEP problem, formulated as MILP and MIQP optimization problems, respectively.

Let $\boldsymbol{x^\mathrm{p}} \in \mathbb{R}^G$ and $\boldsymbol{x^\mathrm{s}} \in \mathbb{R}^N$ denote the installed capacities of the generators and the energy storage systems, respectively.
The power output of the generators at time $t$ is denoted by $\boldsymbol{p}_{t} \in \mathbb{R}^G$,
and the non-supplied energy demand in the same period is denoted by $d^\mathrm{ns}_t$.
Moreover, the state of charge of the energy storage systems at time $t$ is denoted by $\boldsymbol{s}_{t} \in \mathbb{R}^N$.
The initial state of charge of the storage systems is denoted by $\boldsymbol{S^0} \in \mathbb{R}^N$,
while their charging and discharging efficiencies are denoted by $\boldsymbol{\eta^\mathrm{c}} \in \mathbb{R}^N$ and $\boldsymbol{\eta^\mathrm{d}} \in \mathbb{R}^N$, respectively.
The power charged to and discharged from the energy storage systems are $\boldsymbol{p}^{\boldsymbol{\mathrm{c}}}_t \in \mathbb{R}^N$ and $\boldsymbol{p}^{\boldsymbol{\mathrm{d}}}_t \in \mathbb{R}^N$, respectively.
The maximum and minimum charging power limits are denoted by $\overline{\boldsymbol{P}}^{\boldsymbol{\mathrm{c}}} \in \mathbb{R}^N$ and $\underline{\boldsymbol{P}}^{\boldsymbol{\mathrm{c}}} \in \mathbb{R}^N$, respectively.
Similarly, the maximum and minimum discharging power limits are denoted by $\overline{\boldsymbol{P}}^{\boldsymbol{\mathrm{d}}} \in \mathbb{R}^N$ and $\underline{\boldsymbol{P}}^{\boldsymbol{\mathrm{d}}} \in \mathbb{R}^N$, respectively.
The binary variables in $\boldsymbol{b} \in \mathbb{R}^{G+N}$ are used to enforce that the installed capacities are either 0
or within their upper $\left(\overline{\boldsymbol{X}} \in \mathbb{R}^{G+N}\right)$ and lower $\left(\underline{\boldsymbol{X}} \in \mathbb{R}^{G+N}\right)$ limits.

To streamline notation, we define the vectors
$\boldsymbol{x} \coloneqq \left[ {\boldsymbol{x}^{\boldsymbol{\mathrm{p}}}}^\top, {\boldsymbol{x}^{\boldsymbol{\mathrm{s}}}}^\top \right]^\top$,
$\boldsymbol{p}^{\boldsymbol{\mathrm{s}}}_{t} \coloneqq \left[ {\boldsymbol{p}^{\boldsymbol{\mathrm{c}}}_{t}}^\top, {\boldsymbol{p}^{\boldsymbol{\mathrm{d}}}_{t}}^\top \right]^\top$,
$\overline{\boldsymbol{P}}^{\boldsymbol{\mathrm{s}}} \coloneqq \left[{\overline{\boldsymbol{P}}^{\boldsymbol{\mathrm{c}}}}^\top, {\overline{\boldsymbol{P}}^{\boldsymbol{\mathrm{d}}}}^\top\right]^\top$,
and
$\underline{\boldsymbol{P}}^{\boldsymbol{\mathrm{s}}} \coloneqq \left[{\underline{\boldsymbol{P}}^{\boldsymbol{\mathrm{c}}}}^\top, {\underline{\boldsymbol{P}}^{\boldsymbol{\mathrm{d}}}}^\top\right]^\top$.

The set of decision variables in the full-scale GEP optimization model is denoted by $\boldsymbol{z}$, and is defined as $\boldsymbol{z} \coloneqq \left\{\boldsymbol{x}, \boldsymbol{b}, \boldsymbol{p}_{t}, \boldsymbol{s}_{t}, \boldsymbol{p}^{\boldsymbol{\mathrm{s}}}_{t}, d^\mathrm{ns}_t \; | \; t \in \boldsymbol{T} \right\}$,
with all variables in $\boldsymbol{z}$ assumed to take non-negative values.

Then, the \textbf{full-scale MILP model} for the GEP problem is defined over a discrete set of time periods $t \in \boldsymbol{T}$, with sampling time $\Delta$, as follows:
\begin{subequations}\label{MILP_full_scale_investment_model}
\begin{align}
\min_{\boldsymbol{z}} \quad & J(\boldsymbol{z}) \coloneqq \boldsymbol{C^\mathrm{inv}}^\top \boldsymbol{x} + \sum_{t \in \boldsymbol{T}} \left( \boldsymbol{C^\mathrm{p}}^\top \boldsymbol{p}_{t} \, \Delta + C^\mathrm{ns} \, d^\mathrm{ns}_t\right) \label{MILP_full_scale_investment_model_obj}\\
\textrm{s.t.} \quad & \mathbf{1}_G^\top \, \boldsymbol{p}_{t} \, \Delta + \mathbf{1}_N^\top \left(\boldsymbol{p}^{\boldsymbol{\mathrm{d}}}_{t} - \boldsymbol{p}^{\boldsymbol{\mathrm{c}}}_{t}\right) \Delta + d^\mathrm{ns}_t = D_t, \quad \forall t, \label{MILP_full_scale_investment_model_power_balance}\\
& \boldsymbol{s}_{t+1} = \boldsymbol{s}_{t} + \left( \boldsymbol{\eta^\mathrm{c}} \circ \boldsymbol{p}^{\boldsymbol{\mathrm{c}}}_{t} - \boldsymbol{\eta^\mathrm{d}} \circ \boldsymbol{p}^{\boldsymbol{\mathrm{d}}}_{t} \right) \Delta, \quad \forall t \in \left\{0,\dots,\left|\boldsymbol{T}\right| - 1\right\}, \label{MILP_full_scale_investment_model_storage_state}\\
& \boldsymbol{s}_{0} = \boldsymbol{S^0}, \label{MILP_full_scale_investment_model_storage_init}\\
& \boldsymbol{\underline{P}}^{\boldsymbol{\mathrm{s}}} \leq \boldsymbol{p}^{\boldsymbol{\mathrm{s}}}_{t} \leq \boldsymbol{\overline{P}}^{\boldsymbol{\mathrm{s}}}, \quad \forall t, \label{MILP_full_scale_investment_model_storage_power_limits}\\
& 0 \leq \boldsymbol{s}_{t} \leq \boldsymbol{x}^{\boldsymbol{\mathrm{s}}} \, \Delta, \quad \forall t, \label{MILP_full_scale_investment_model_sto_limits}\\
& 0 \leq \boldsymbol{p}_{t} \leq \boldsymbol{F}_{t} \circ \boldsymbol{x}^{\boldsymbol{\mathrm{p}}}, \quad \forall t, \label{MILP_full_scale_investment_model_gen_limits}\\
& \boldsymbol{b} \circ \underline{\boldsymbol{X}} \leq \boldsymbol{x} \leq \boldsymbol{b} \circ \overline{\boldsymbol{X}}, \label{MILP_full_scale_investment_model_inv_limits}\\
& \boldsymbol{b} \in \{0,1\}^{G + N}. \label{MILP_full_scale_investment_model_binary}
\end{align}
\end{subequations}
In \eqref{MILP_full_scale_investment_model}, the objective function $J(\boldsymbol{z})$, as defined in \eqref{MILP_full_scale_investment_model_obj}, captures the investment and operational costs incurred over the GEP horizon.
The energy balance is enforced by \eqref{MILP_full_scale_investment_model_power_balance}.
The dynamics of the storage systems are governed by \eqref{MILP_full_scale_investment_model_storage_state} and \eqref{MILP_full_scale_investment_model_storage_init},
while the charging and discharging power limits are enforced by \eqref{MILP_full_scale_investment_model_storage_power_limits}.
The capacity limits for generation and storage are imposed by \eqref{MILP_full_scale_investment_model_gen_limits} and \eqref{MILP_full_scale_investment_model_sto_limits}, respectively.
Finally, the installed capacities are constrained by \eqref{MILP_full_scale_investment_model_inv_limits}.

To accommodate a broader class of GEP formulations, we extend the MILP model in \eqref{MILP_full_scale_investment_model} by adding to the objective function a quadratic regularization term that penalizes deviations of selected variables from predefined reference values. This yields the following \textbf{full-scale MIQP model}:
\begin{subequations}\label{MIQP_full_scale_investment_model}
\begin{align}
\min_{\boldsymbol{z}} \quad & \bar{J}(\boldsymbol{z}) \coloneqq J\left(\boldsymbol{z}\right) + \sum_{t \in \boldsymbol{T}} \left\|\mathbf{A} \boldsymbol{z}^{\boldsymbol{\mathrm{op}}}_t - \boldsymbol{Z}^{\boldsymbol{\mathrm{ref}}}_t\right\|_2^2 \label{MIQP_full_scale_investment_model_obj}\\
\textrm{s.t.} \quad & \eqref{MILP_full_scale_investment_model_power_balance}-\eqref{MILP_full_scale_investment_model_binary}. \nonumber
\end{align}
\end{subequations}
In \eqref{MIQP_full_scale_investment_model_obj},
$\| \cdot \|_2$ denotes the Euclidean (or $\ell^2$) norm,
$\mathbf{A} \in \mathbb{R}^{R \times \left(G + 3 N + 1\right)}$,
$\boldsymbol{z}^{\boldsymbol{\mathrm{op}}}_t \coloneqq \left[ {\boldsymbol{p}}^\top_{t}, d^\mathrm{ns}_t, \boldsymbol{s}^\top_{t}, {\boldsymbol{p}^{\boldsymbol{\mathrm{s}}}_t}^\top \right]^\top \in \mathbb{R}^{G + 3 N + 1}$ collects the operational variables in \eqref{MILP_full_scale_investment_model}.
Moreover, $\boldsymbol{Z}^{\boldsymbol{\mathrm{ref}}}_t \in \mathbb{R}^{R}$ collects the input reference values at time $t$.

For example, by taking 
$\mathbf{A} = 
    \begin{bmatrix}
    \mathbf{0}_{N \times G + 1},
    \mathbf{I}_{N},
    \mathbf{0}_{N \times 2 N}
    \end{bmatrix}$,
we penalize deviations of the storage states $\boldsymbol{s}_{t}$ from their reference values $\boldsymbol{Z}^{\boldsymbol{\mathrm{ref}}}_t$:
\begin{equation}
\label{storage_quadratic_penalty}
    \sum_{t \in \boldsymbol{T}} \left\|\mathbf{A} \boldsymbol{z}^{\boldsymbol{\mathrm{op}}}_t - \boldsymbol{Z}^{\boldsymbol{\mathrm{ref}}}_t\right\|_2^2 = \sum_{t \in \boldsymbol{T}} \left\|\boldsymbol{s}_{t} - \boldsymbol{Z}^{\boldsymbol{\mathrm{ref}}}_t\right\|_2^2.
\end{equation}
Quadratic regularization terms of the form \eqref{storage_quadratic_penalty} are widely used in storage operations \cite{santosuosso2024stochastic}.
Whereas the MILP model \eqref{MILP_full_scale_investment_model} may induce full utilization of the storage flexibility, this regularization term promotes a more conservative operation, preserving flexibility capacity for future unforeseen events.

\subsection{The Aggregated Optimization Models}
\label{subsec:aggregated_models}
This subsection presents the aggregated GEP optimization models.

Mixed-integer formulations of the GEP problem are strongly NP-hard \cite{goderbauer2019synthesis}.
As the number of binary variables and time periods increases, the full-scale models \eqref{MILP_full_scale_investment_model} and \eqref{MIQP_full_scale_investment_model} may become computationally expensive or even intractable.
To mitigate this, TSA can be used to derive aggregated models defined over a reduced set of representative time periods (or clusters), denoted by $\boldsymbol{K}$ and indexed by $k$.
When $|\boldsymbol{K}| \ll |\boldsymbol{T}|$, solving the aggregated model in place of its full-scale counterpart yields significant computational savings.

Let $\boldsymbol{T}_k$ denote the set of time periods in $\boldsymbol{T}$ assigned to the $k$-th cluster, and define $T_k \coloneqq \left|\boldsymbol{T}_k\right|$ as the number of time periods in the $k$-th cluster.

\textbf{Assumption I}: The aggregated models are constructed by clustering the original set of time periods $\boldsymbol{T}$ while preserving temporal chronology, i.e., for any $t' \in \boldsymbol{T}_{k'}$ and $t'' \in \boldsymbol{T}_{k''}$ with $k' < k''$, it holds that $t' < t''$.

To retain consistency with the full-scale models, the investment and storage state variables in the following aggregated models are defined as 

\noindent
\begin{minipage}{0.3\textwidth}
\begin{equation}\label{agg_vars:inv}
    \boldsymbol{\hat{x}} \coloneqq \boldsymbol{x},
\end{equation}
\end{minipage}%
\hfill
\begin{minipage}{0.3\textwidth}
\begin{equation}\label{agg_vars:inv_bin}
    \boldsymbol{\hat{b}} \coloneqq \boldsymbol{b},
\end{equation}
\end{minipage}%
\hfill
\begin{minipage}{0.4\textwidth}
\begin{equation}\label{agg_vars:sto}
    \boldsymbol{\hat{s}}_{k} \coloneqq \boldsymbol{s}_{\mathrm{min}(\boldsymbol{T}_k)}, \; \forall k,
\end{equation}
\end{minipage}
where $\mathrm{min}\left(\boldsymbol{T}_k\right)$ denotes the first element in $\boldsymbol{T}_k$.

The aggregated variables corresponding to the remaining full-scale variables in \eqref{MILP_full_scale_investment_model} and \eqref{MIQP_full_scale_investment_model} are defined as

\noindent
\begin{minipage}{0.5\textwidth}
\begin{equation}\label{agg_vars:nsdemand}
    \hat{d}^{\mathrm{ns}}_k \coloneqq \frac{1}{T_k} \sum_{t \in \boldsymbol{T}_k} d_t^{\mathrm{ns}}, \; \forall k,
\end{equation}
\end{minipage}%
\hfill
\begin{minipage}{0.5\textwidth}
\begin{equation}\label{agg_vars:powergen}
    \boldsymbol{\hat{p}}_{k} \coloneqq \frac{1}{T_k} \sum_{t \in \boldsymbol{T}_k} \boldsymbol{p}_{t}, \; \forall k,
\end{equation}
\end{minipage}

\noindent
\begin{minipage}{0.5\textwidth}
\begin{equation}\label{agg_vars:stochr}
    \boldsymbol{\hat{p}}^{\boldsymbol{\mathrm{c}}}_{k} \coloneqq \frac{1}{T_k} \sum_{t \in \boldsymbol{T}_k} \boldsymbol{p}^{\boldsymbol{\mathrm{c}}}_{t}, \; \forall k,
\end{equation}
\end{minipage}%
\hfill
\begin{minipage}{0.5\textwidth}
\begin{equation}\label{agg_vars:stodischr}
    \boldsymbol{\hat{p}}^{\boldsymbol{\mathrm{d}}}_{k} \coloneqq \frac{1}{T_k} \sum_{t \in \boldsymbol{T}_k} \boldsymbol{p}^{\boldsymbol{\mathrm{d}}}_{t}, \; \forall k.
\end{equation}
\end{minipage}

\noindent Moreover, we denote by $\boldsymbol{\hat{p}}^{\boldsymbol{\mathrm{s}}}_{k} \coloneqq \left[ {\boldsymbol{\hat{p}}^{\boldsymbol{\mathrm{c}}}_{k}}^\top, {\boldsymbol{\hat{p}}^{\boldsymbol{\mathrm{d}}}_{k}}^\top \right]^\top$ the aggregated vector of storage charging and discharging power,
and we group the decision variables in \eqref{agg_vars:inv}--\eqref{agg_vars:stodischr} into the set $\boldsymbol{\hat{z}}$, defined as $\boldsymbol{\hat{z}} \coloneqq \left\{\boldsymbol{\hat{x}}, \boldsymbol{\hat{b}}, \boldsymbol{\hat{p}}_{k}, \boldsymbol{\hat{s}}_{k}, \boldsymbol{\hat{p}}^{\boldsymbol{\mathrm{s}}}_{k}, \hat{d}^\mathrm{ns}_k \; | \; k \in \boldsymbol{K} \right\}$.

Then, the \textbf{aggregated MILP model}, corresponding to the full-scale MILP model \eqref{MILP_full_scale_investment_model}, is defined over the set of representative periods $\boldsymbol{K}$ as
\begin{subequations}\label{MILP_aggregated_investment_model}
\begin{align}
\min_{\boldsymbol{\hat{z}}} \quad & \hat{J}(\boldsymbol{\hat{z}}) \coloneqq {\boldsymbol{C}^{\boldsymbol{\mathrm{inv}}}}^\top \boldsymbol{\hat{x}} + \sum_{k \in \boldsymbol{K}} T_k \left( {\boldsymbol{C}^{\boldsymbol{\mathrm{p}}}}^\top \boldsymbol{\hat{p}}_{k} \, \Delta + C^\mathrm{ns} \, \hat{d}^\mathrm{ns}_k\right) \label{MILP_aggregated_investment_model_obj}\\
\textrm{s.t.} \quad & \mathbf{1}_G^\top \, \boldsymbol{\hat{p}}_{k} \, \Delta + \mathbf{1}_N^\top \left( \boldsymbol{\hat{p}}^{\boldsymbol{\mathrm{d}}}_{k} - \boldsymbol{\hat{p}}^{\boldsymbol{\mathrm{c}}}_{k} \right) \Delta + \hat{d}^\mathrm{ns}_k = \frac{1}{T_k} \sum_{t \in \boldsymbol{T}_k} D_t, \quad \forall k, \label{MILP_aggregated_investment_model_power_balance}\\
&\boldsymbol{\hat{s}}_{k+1} = \boldsymbol{\hat{s}}_{k} + \left( {\boldsymbol{\eta^\mathrm{c}}} \circ \boldsymbol{\hat{p}}^{\boldsymbol{\mathrm{c}}}_{k} - {\boldsymbol{\eta^\mathrm{d}}} \circ \boldsymbol{\hat{p}}^{\boldsymbol{\mathrm{d}}}_{k} \right) T_k \, \Delta, \;\, \forall k \in \left\{0,\dots,\left|\boldsymbol{K}\right|-1\right\}, \label{MILP_aggregated_investment_model_storage_state}\\
& \boldsymbol{\hat{s}}_{0} = \boldsymbol{S^0}, \label{MILP_aggregated_investment_model_storage_init}\\
& \boldsymbol{\underline{P}}^{\boldsymbol{\mathrm{s}}} \leq \boldsymbol{\hat{p}}^{\boldsymbol{\mathrm{s}}}_{k} \leq \boldsymbol{\overline{P}}^{\boldsymbol{\mathrm{s}}}, \quad \forall k, \label{MILP_aggregated_investment_model_storage_power_limits}\\
& 0 \leq \boldsymbol{\hat{s}}_{k} \leq \boldsymbol{\hat{x}}^{\boldsymbol{\mathrm{s}}} \, \Delta, \quad \forall k, \label{MILP_aggregated_investment_model_sto_limits}\\
& 0 \leq \boldsymbol{\hat{p}}_{k} \leq \boldsymbol{\hat{x}}^{\boldsymbol{\mathrm{p}}} \circ \left(\frac{1}{T_k} \sum_{t \in \boldsymbol{T}_k} \boldsymbol{F}_{t}\right), \quad \forall k, \label{MILP_aggregated_investment_model_gen_limits}\\
& \boldsymbol{\hat{b}} \circ \underline{\boldsymbol{X}} \leq \boldsymbol{\hat{x}} \leq \boldsymbol{\hat{b}} \circ \overline{\boldsymbol{X}}, \label{MILP_aggregated_investment_model_inv_limits}\\
& \boldsymbol{\hat{b}} \in \{0,1\}^{G + N}. \label{MILP_aggregated_investment_model_binary}
\end{align}
\end{subequations}
In \eqref{MILP_aggregated_investment_model}, the objective function \eqref{MILP_aggregated_investment_model_obj} and the constraints \eqref{MILP_aggregated_investment_model_power_balance}--\eqref{MILP_aggregated_investment_model_binary} represent the aggregated counterparts of \eqref{MILP_full_scale_investment_model_obj} and \eqref{MILP_full_scale_investment_model_power_balance}--\eqref{MILP_full_scale_investment_model_binary}, respectively.

Similarly, the \textbf{aggregated MIQP model}, corresponding to the full-scale MIQP model \eqref{MIQP_full_scale_investment_model}, is defined over the set of representative periods $\boldsymbol{K}$ as
\begin{subequations}\label{MIQP_aggregated_investment_model}
\begin{align}
\min_{\boldsymbol{\hat{z}}} \quad & \hat{\bar{J}}(\boldsymbol{\hat{z}}) \coloneqq \hat{J}(\boldsymbol{\hat{z}}) + \sum_{k \in \boldsymbol{K}} T_k \left\|\mathbf{A} \boldsymbol{\hat{z}}^{\boldsymbol{\mathrm{op}}}_k - \frac{1}{T_k} \sum_{t \in \boldsymbol{T}_k} \boldsymbol{Z}^{\boldsymbol{\mathrm{ref}}}_t \right\|_2^2 \label{MIQP_aggregated_investment_model_obj}\\
\textrm{s.t.} \quad & \eqref{MILP_aggregated_investment_model_power_balance}-\eqref{MILP_aggregated_investment_model_binary}. \nonumber
\end{align}
\end{subequations}
In \eqref{MIQP_aggregated_investment_model}, the objective function \eqref{MIQP_aggregated_investment_model_obj} is the aggregated counterpart of \eqref{MIQP_full_scale_investment_model_obj},
and $\boldsymbol{\hat{z}}^{\boldsymbol{\mathrm{op}}}_k$ is defined as $\boldsymbol{\hat{z}}^{\boldsymbol{\mathrm{op}}}_k \coloneqq \left[ {\boldsymbol{\hat{p}}}^\top_{k}, \hat{d}^\mathrm{ns}_k, \boldsymbol{\hat{s}}^\top_{k}, {\boldsymbol{\hat{p}}^{\boldsymbol{\mathrm{s}}}_{k}}^\top \right]^\top \in \mathbb{R}^{G + 3 N + 1}$.

\subsection{Properties of the Aggregated Optimization Models}
\label{subsec:properties_aggregated_models}
This subsection analyzes the theoretical properties of the aggregated models in Subsection~\ref{subsec:aggregated_models} relative to their full-scale counterparts in Subsection~\ref{subsec:full_scale_models}.

\begin{lemma}\label{lemma:feasible_aggregated_sol}
Let $\boldsymbol{z}$ be a feasible solution to the full-scale MILP model \eqref{MILP_full_scale_investment_model},
or equivalently, to the full-scale MIQP model \eqref{MIQP_full_scale_investment_model}.
Let $\boldsymbol{\hat{z}}$ be derived from $\boldsymbol{z}$ according to \eqref{agg_vars:inv}--\eqref{agg_vars:stodischr}.
Then, under Assumption~I, $\boldsymbol{\hat{z}}$ is a feasible solution to both the aggregated models \eqref{MILP_aggregated_investment_model} and \eqref{MIQP_aggregated_investment_model}.
\end{lemma}
\begin{proof}
From \eqref{agg_vars:nsdemand}–\eqref{agg_vars:stodischr}, substituting $\boldsymbol{\hat{z}}$ into the energy balance constraints \eqref{MILP_aggregated_investment_model_power_balance} of the aggregated models \eqref{MILP_aggregated_investment_model} and \eqref{MIQP_aggregated_investment_model} yields:
\begin{equation}\label{lemma:feasible_aggregated_sol_proof1}
    \sum_{t \in \boldsymbol{T}_k} \left( \mathbf{1}_G^\top \, \boldsymbol{p}_{t} \Delta + \boldsymbol{1}^\top_N \left(\boldsymbol{p}^{\boldsymbol{\mathrm{d}}}_{t} - \boldsymbol{p}^{\boldsymbol{\mathrm{c}}}_{t}\right) \Delta + d^\mathrm{ns}_t \right) = \sum_{t \in \boldsymbol{T}_k} D_t, \quad \forall k.
\end{equation}

From \eqref{agg_vars:sto}, \eqref{agg_vars:stochr}, and \eqref{agg_vars:stodischr}, substituting $\boldsymbol{\hat{z}}$ into the storage dynamics \eqref{MILP_aggregated_investment_model_storage_state} of the aggregated models \eqref{MILP_aggregated_investment_model} and \eqref{MIQP_aggregated_investment_model} yields:
\begin{equation}\label{lemma:feasible_aggregated_sol_proof2}
    \boldsymbol{s}_{\mathrm{min}(\boldsymbol{T}_{k+1})} = \boldsymbol{s}_{\mathrm{min}(\boldsymbol{T}_k)} + \sum_{t \in \boldsymbol{T}_k} \left( \boldsymbol{\eta^\mathrm{c}} \circ \boldsymbol{p}^{\boldsymbol{\mathrm{c}}}_{t} \right) \Delta - \sum_{t \in \boldsymbol{T}_k} \left( \boldsymbol{\eta^\mathrm{d}} \circ \boldsymbol{p}^{\boldsymbol{\mathrm{d}}}_{t} \right) \Delta, \quad \forall k.
\end{equation}

From \eqref{agg_vars:stochr} and \eqref{agg_vars:stodischr}, substituting $\boldsymbol{\hat{z}}$ into the storage power limit constraints \eqref{MILP_aggregated_investment_model_storage_power_limits} of the aggregated models \eqref{MILP_aggregated_investment_model} and \eqref{MIQP_aggregated_investment_model} yields:
\begin{equation}\label{lemma:feasible_aggregated_sol_proof3}
    \boldsymbol{\underline{P}}^{\boldsymbol{\mathrm{s}}} \leq \frac{1}{T_k} \sum_{t \in \boldsymbol{T}_k} \boldsymbol{p}^{\boldsymbol{\mathrm{s}}}_{t} \leq \boldsymbol{\overline{P}}^{\boldsymbol{\mathrm{s}}}, \quad \forall k.
\end{equation}

From \eqref{agg_vars:inv}, \eqref{agg_vars:sto}, and \eqref{agg_vars:powergen}, substituting $\boldsymbol{\hat{z}}$ into the constraints \eqref{MILP_aggregated_investment_model_sto_limits} and \eqref{MILP_aggregated_investment_model_gen_limits} of the aggregated models \eqref{MILP_aggregated_investment_model} and \eqref{MIQP_aggregated_investment_model}, enforcing the storage capacity and power generation limits, yields:
\begin{align}
    & 0 \leq \boldsymbol{s}_{\mathrm{min}(\boldsymbol{T}_k)} \leq \boldsymbol{x}^{\boldsymbol{\mathrm{s}}} \, \Delta, \quad \forall k, \label{lemma:feasible_aggregated_sol_proof4}\\
    & 0 \leq \sum_{t \in \boldsymbol{T}_k} \boldsymbol{p}_{t} \leq \sum_{t \in \boldsymbol{T}_k} \left( \boldsymbol{x}^{\boldsymbol{\mathrm{p}}} \circ \boldsymbol{F}_{t} \right), \quad \forall k. \label{lemma:feasible_aggregated_sol_proof5}
\end{align}

Under Assumption~I, the individual constraints \eqref{MILP_full_scale_investment_model_power_balance}, \eqref{MILP_full_scale_investment_model_storage_state}, \eqref{MILP_full_scale_investment_model_storage_power_limits}, \eqref{MILP_full_scale_investment_model_sto_limits} and \eqref{MILP_full_scale_investment_model_gen_limits} of the full-scale models \eqref{MILP_full_scale_investment_model} and \eqref{MIQP_full_scale_investment_model} imply the aggregate constraints \eqref{lemma:feasible_aggregated_sol_proof1}, \eqref{lemma:feasible_aggregated_sol_proof2}, \eqref{lemma:feasible_aggregated_sol_proof3}, \eqref{lemma:feasible_aggregated_sol_proof4}, and \eqref{lemma:feasible_aggregated_sol_proof5}, respectively.
Moreover, it holds that $\boldsymbol{\hat{s}}_{0} = \boldsymbol{s}_{\mathrm{min}(\boldsymbol{T}_0)}$ and $\boldsymbol{s}_{\mathrm{min}(\boldsymbol{T}_0)} \coloneqq \boldsymbol{s}_{0}$, implying that the aggregate constraint \eqref{MILP_aggregated_investment_model_storage_init} is equivalent to the full-scale constraint \eqref{MILP_full_scale_investment_model_storage_init}.
Finally, from \eqref{agg_vars:inv} and \eqref{agg_vars:inv_bin}, substituting $\boldsymbol{\hat{z}}$ into the aggregate constraints \eqref{MILP_aggregated_investment_model_inv_limits} and \eqref{MILP_aggregated_investment_model_binary} results in constraints equivalent to \eqref{MILP_full_scale_investment_model_inv_limits} and \eqref{MILP_full_scale_investment_model_binary} of the full-scale models \eqref{MILP_full_scale_investment_model} and \eqref{MIQP_full_scale_investment_model}.

It follows that $\boldsymbol{\hat{z}}$, derived via \eqref{agg_vars:inv}–\eqref{agg_vars:stodischr} from a feasible solution $\boldsymbol{z}$ of the full-scale MILP model \eqref{MILP_full_scale_investment_model} or, equivalently, the MIQP model \eqref{MIQP_full_scale_investment_model}, constitutes a feasible solution to both the aggregated models \eqref{MILP_aggregated_investment_model} and \eqref{MIQP_aggregated_investment_model}.
\end{proof}

In essence, the theoretical result established in Lemma~\ref{lemma:feasible_aggregated_sol} implies that for every feasible solution of the full-scale models presented in Subsection~\ref{subsec:full_scale_models}, there exists a corresponding feasible solution to the aggregated models introduced in Subsection~\ref{subsec:aggregated_models}.
Thus, the feasible region of the aggregated models \eqref{MILP_aggregated_investment_model} and \eqref{MIQP_aggregated_investment_model} constitutes a relaxation of that of the full-scale models \eqref{MILP_full_scale_investment_model} and \eqref{MIQP_full_scale_investment_model}.
This result can be leveraged, as detailed in the following proposition, to demonstrate that the optimal objective function values of the aggregated models \eqref{MILP_aggregated_investment_model} and \eqref{MIQP_aggregated_investment_model} always provide valid lower bounds on the optimal objective function values of their full-scale counterparts, \eqref{MILP_full_scale_investment_model} and \eqref{MIQP_full_scale_investment_model}, respectively.

\begin{proposition}\label{prop:aggregated_bound}
Let $\boldsymbol{z}^\star$ and $\boldsymbol{\hat{z}}^\star$ denote optimal solutions of the full-scale MILP model \eqref{MILP_full_scale_investment_model} and the aggregated MILP model \eqref{MILP_aggregated_investment_model}, respectively.
Similarly, let $\boldsymbol{\bar{z}}^\star$ and $\boldsymbol{\hat{\bar{z}}}^\star$ denote optimal solutions of the full-scale MIQP model \eqref{MIQP_full_scale_investment_model} and the aggregated MIQP model \eqref{MIQP_aggregated_investment_model}, respectively.
Then, the following inequalities hold:

\noindent
\begin{minipage}{0.5\textwidth}
\begin{equation}\label{prop:aggregated_bound_result_MILP}
    \hat{J}\left(\boldsymbol{\hat{z}}^\star\right) \leq J\left(\boldsymbol{z}^\star\right),
\end{equation}
\end{minipage}%
\hfill
\begin{minipage}{0.5\textwidth}
\begin{equation}\label{prop:aggregated_bound_result_MIQP}
    \hat{\bar{J}}\left(\boldsymbol{\hat{\bar{z}}}^\star\right) \leq \bar{J}\left(\boldsymbol{\bar{z}}^\star\right).
\end{equation}
\end{minipage}
\end{proposition}

\begin{proof}
We begin by demonstrating the inequality \eqref{prop:aggregated_bound_result_MILP} for the aggregated model \eqref{MILP_aggregated_investment_model} and its full-scale counterpart \eqref{MILP_full_scale_investment_model}.

By Lemma~\ref{lemma:feasible_aggregated_sol},
for any feasible solution $\boldsymbol{z}$ of the full-scale MILP model~\eqref{MILP_full_scale_investment_model},
there exists a corresponding feasible solution $\boldsymbol{\hat{z}}$ of the aggregated MILP model \eqref{MIQP_aggregated_investment_model},
constructed according to \eqref{agg_vars:inv}--\eqref{agg_vars:stodischr}.
Then, substituting $\boldsymbol{\hat{z}}$ into the aggregate objective function $\hat{J}(\boldsymbol{\hat{z}})$ of the aggregated model \eqref{MILP_aggregated_investment_model} yields:
\begin{align*}
    \hat{J}(\boldsymbol{\hat{z}}) = & \; {\boldsymbol{C}^{\boldsymbol{\mathrm{inv}}}}^\top \boldsymbol{x} + \sum_{k \in \boldsymbol{K}} T_k \left( \frac{1}{T_k} \sum_{t \in \boldsymbol{T}_k} {\boldsymbol{C}^{\boldsymbol{\mathrm{p}}}}^\top \boldsymbol{p}_{t} \, \Delta + \frac{1}{T_k} \sum_{t \in \boldsymbol{T}_k} C^\mathrm{ns} d^\mathrm{ns}_t \right) \\
    = & \; {\boldsymbol{C}^{\boldsymbol{\mathrm{inv}}}}^\top \boldsymbol{x} +  \sum_{k \in \boldsymbol{K}} \sum_{t \in \boldsymbol{T}_k} \left(\, {\boldsymbol{C}^{\boldsymbol{\mathrm{p}}}}^\top \boldsymbol{p}_{t} \, \Delta + C^\mathrm{ns} d^\mathrm{ns}_t \right) \\
    = & \; {\boldsymbol{C}^{\boldsymbol{\mathrm{inv}}}}^\top \boldsymbol{x} +  \sum_{t \in \boldsymbol{T}} \left( {\boldsymbol{C}^{\boldsymbol{\mathrm{p}}}}^\top \boldsymbol{p}_{t} \, \Delta + C^\mathrm{ns} d^\mathrm{ns}_t \right) \\
    = & \; J(\boldsymbol{z}),
\end{align*}
where $J(\boldsymbol{z})$ is defined as in \eqref{MILP_full_scale_investment_model_obj}.

It follows that every feasible solution $\boldsymbol{z}$ of the full-scale model~\eqref{MILP_full_scale_investment_model} maps to a feasible solution $\boldsymbol{\hat{z}}$ of the aggregated model~\eqref{MILP_aggregated_investment_model},
yielding the same objective function value.
In particular, this holds for any full-scale optimum $\boldsymbol{z}^\star$. 
Thus, the optimal objective function value of the aggregated model~\eqref{MILP_aggregated_investment_model} provides a valid lower bound on that of the full-scale model~\eqref{MILP_full_scale_investment_model}, i.e., $\hat{J}\left(\boldsymbol{\hat{z}}^\star\right) \leq J\left(\boldsymbol{z}^\star\right)$.

An analogous approach can be used to derive the result in \eqref{prop:aggregated_bound_result_MIQP} for the MIQP model.
By Lemma~\ref{lemma:feasible_aggregated_sol}, for any feasible solution $\boldsymbol{z}$ of the full-scale MIQP model~\eqref{MIQP_full_scale_investment_model}, there exists a corresponding feasible solution $\boldsymbol{\hat{z}}$ of the aggregated MIQP model~\eqref{MIQP_aggregated_investment_model}, constructed according to \eqref{agg_vars:inv}–\eqref{agg_vars:stodischr}. Moreover, evaluating $\boldsymbol{\hat{z}}$ in the objective function $\hat{\bar{J}}(\boldsymbol{\hat{z}})$ of the aggregated model~\eqref{MIQP_aggregated_investment_model} yields:
\begin{align*}
    \hat{\bar{J}}(\boldsymbol{\hat{z}}) = & \; J(\boldsymbol{z}) + \sum_{k \in \boldsymbol{K}} T_k \left\|\frac{1}{T_k} \sum_{t \in \boldsymbol{T}_k} \left(\mathbf{A} \boldsymbol{z}^{\boldsymbol{\mathrm{op}}}_t - \boldsymbol{Z}^{\boldsymbol{\mathrm{ref}}}_t \right) \right\|_2^2 \\
    \leq & \; J(\boldsymbol{z}) + \sum_{t \in \boldsymbol{T}} \left\| \left(\mathbf{A} \, \boldsymbol{z}^{\boldsymbol{\mathrm{op}}}_t - \boldsymbol{Z}^{\boldsymbol{\mathrm{ref}}}_t \right) \right\|_2^2 = \bar{J}(\boldsymbol{z}),
\end{align*}
where the inequality follows directly from Jensen’s inequality.
\end{proof}

\subsection{A Time Series Aggregation Method with a Performance Guarantee}
\label{subsec:performance_guaranteed_TSA}
Building upon the theoretical results from Subsection~\ref{subsec:properties_aggregated_models},
we propose an iterative TSA method that offers a formal performance guarantee in the form of objective function bounds.
At each iteration, the method computes both upper and lower bounds on the optimal objective value of the full-scale MILP or the full-scale MIQP model \eqref{MILP_full_scale_investment_model} and \eqref{MIQP_full_scale_investment_model}, respectively,
by leveraging the solutions from their aggregated counterparts \eqref{MILP_aggregated_investment_model} and \eqref{MIQP_aggregated_investment_model}, respectively.

From Proposition~\ref{prop:aggregated_bound}, the aggregated models in Subsection~\ref{subsec:aggregated_models} provide valid lower bounds on the optimal objective values of their full-scale counterparts.
Moreover, we remark that the full-scale models in Subsection~\ref{subsec:full_scale_models} are two-stage optimization problems,
where the investment variable $\boldsymbol{x}$ and the associated binary variable $\boldsymbol{b}$ are the first-stage decisions,
while the operational variables in $\boldsymbol{z}^{\boldsymbol{\mathrm{op}}}_t$ are the second-stage decisions at time period $t$.
By fixing the first-stage decisions to the corresponding values obtained from the aggregated models, the full-scale MILP and MIQP models reduce to the convex models:
\noindent
\begin{minipage}{0.5\textwidth}
    \begin{subequations}\label{LP_MILP_full_scale_investment_model}
    \begin{align}
    \min_{\boldsymbol{z}} \quad & J(\boldsymbol{z})\\
    \textrm{s.t.} \quad & \eqref{MILP_full_scale_investment_model_power_balance}-\eqref{MILP_full_scale_investment_model_inv_limits}, \nonumber \\
    & \boldsymbol{b} = \boldsymbol{\hat{b}}^\star.
    \end{align}
    \end{subequations}
\end{minipage}
\hfill
\begin{minipage}{0.5\textwidth}
    \begin{subequations}\label{LP_MIQP_full_scale_investment_model}
    \begin{align}
    \min_{\boldsymbol{z}} \quad & \bar{J}(\boldsymbol{z})\\
    \textrm{s.t.} \quad & \eqref{MILP_full_scale_investment_model_power_balance}-\eqref{MILP_full_scale_investment_model_inv_limits}, \nonumber \\
    & \boldsymbol{b} = \boldsymbol{\hat{\bar{b}}}^\star.
    \end{align}
    \end{subequations}
\end{minipage}
In \eqref{LP_MILP_full_scale_investment_model} and \eqref{LP_MIQP_full_scale_investment_model},
we denote by $\boldsymbol{\hat{b}}^\star$ and $\boldsymbol{\hat{\bar{b}}}^\star$ the optimal binary variable values obtained by solving the aggregated models \eqref{MILP_aggregated_investment_model} and \eqref{MIQP_aggregated_investment_model}, respectively.
Since \eqref{LP_MILP_full_scale_investment_model} and \eqref{LP_MIQP_full_scale_investment_model} are constrained versions of the full-scale models \eqref{MILP_full_scale_investment_model} and \eqref{MIQP_full_scale_investment_model}, respectively, their optimal objective function values provide upper bounds on those of the full-scale models.
Moreover, as the aggregated models retain full resolution of the binary variables, any solution to the convex models \eqref{LP_MILP_full_scale_investment_model} and \eqref{LP_MIQP_full_scale_investment_model} is also feasible for the corresponding full-scale models.

By increasing the number and/or refining the quality of the representative time periods in the aggregated models,
the proposed TSA method iteratively refines the upper and lower bounds on the optimal objective function values of the full-scale models in Subsection~\ref{subsec:full_scale_models}, as detailed in Algorithm~\ref{alg:TSA_algorithm}.

\begin{algorithm}[!h]
\caption{Time Series Aggregation with a Performance Guarantee}\label{alg:TSA_algorithm}
\begin{algorithmic}[1]
\Require $\Bigl\{\boldsymbol{C}^{\boldsymbol{\mathrm{inv}}}, \boldsymbol{C^\mathrm{p}}, C^\mathrm{ns}, \Delta, \boldsymbol{\eta^\mathrm{c}}, \boldsymbol{\eta^\mathrm{d}}, \boldsymbol{S^0}, \underline{\boldsymbol{P}}^{\boldsymbol{\mathrm{s}}}, \overline{\boldsymbol{P}}^{\boldsymbol{\mathrm{s}}}, \underline{\boldsymbol{X}}, \overline{\boldsymbol{X}}, D_t, \boldsymbol{F}_{t}, \boldsymbol{Z}^{\boldsymbol{\mathrm{ref}}}_t \; | \; t \in \boldsymbol{T} \Bigl\}$, initial number of clusters $K^0$, step size $\rho$, optimality threshold $\epsilon^\mathrm{thr}$, and maximum number of iterations $I$.

\Ensure A feasible solution $\boldsymbol{z}^\mathrm{feasible}$, and the final objective function bounds ${{J^\mathrm{UB}}}^\star$ and ${{J^\mathrm{LB}}}^\star$.

\State \textit{Initialization}: $i \gets 0$; $\epsilon^0 \gets +\infty$;

\While{$\epsilon^i > \epsilon^{\mathrm{thr}}$ and $i \leq I$}

\State Assign the time periods $t \in \boldsymbol{T}$ to $\left\{\boldsymbol{T}_k^i \, | \, k \in \boldsymbol{K}^i\right\}$ using any clustering technique with $K^i$ clusters, under Assumption I;

\If{MILP GEP model}
\State $\left\{ \tilde{J}^\mathrm{LB}, \boldsymbol{\hat{z}}^\star \right\} \gets$ Solve aggregated MILP \eqref{MILP_aggregated_investment_model} for $\left\{\boldsymbol{T}_k^i \, | \, k \in \boldsymbol{K}^i\right\}$;

\State $\left\{\boldsymbol{z}^{\mathrm{feasible}, i}, {\tilde{J}^\mathrm{UB}} \right\} \gets$ Solve \eqref{LP_MILP_full_scale_investment_model} with fixed binaries $\boldsymbol{\hat{b}}^\star$ from $\boldsymbol{\hat{z}}^\star$;
\EndIf

\If{MIQP GEP model}
\State $\left\{ \tilde{J}^\mathrm{LB},\boldsymbol{\hat{\bar{z}}}^\star \right\} \gets$ Solve aggregated MIQP \eqref{MIQP_aggregated_investment_model} for $\left\{\boldsymbol{T}_k^i \, | \, k \in \boldsymbol{K}^i\right\}$;

\State $\left\{\boldsymbol{z}^{\mathrm{feasible}, i}, {\tilde{J}^\mathrm{UB}} \right\} \gets$ Solve \eqref{LP_MIQP_full_scale_investment_model} with fixed binaries $\boldsymbol{\hat{\bar{b}}}^\star$ from $\boldsymbol{\hat{\bar{z}}}^\star$;
\EndIf

\If{$i = 0$} ${J^\mathrm{LB}}^{i + 1} \gets \tilde{J}^\mathrm{LB}$ and ${J^\mathrm{UB}}^{i + 1} \gets \tilde{J}^\mathrm{UB}$;
\Else $\; {J^\mathrm{LB}}^{i + 1} \gets \mathrm{max}\left({J^\mathrm{LB}}^{i}, \tilde{J}^\mathrm{LB}\right)$ and ${J^\mathrm{UB}}^{i + 1} \gets \mathrm{min}\left({J^\mathrm{UB}}^{i}, \tilde{J}^\mathrm{UB}\right)$;
\EndIf
\State $\epsilon^{i+1} \gets \text{Evaluate \eqref{optimality_gap}}$ for ${J^\mathrm{LB}}^{i + 1}$ and ${J^\mathrm{UB}}^{i + 1}$;

\State $K^{i+1} \gets K^i + \rho$;

\State $i \gets i+1$;

\EndWhile

\State $\boldsymbol{z}^\mathrm{feasible} \gets \boldsymbol{z}^{\mathrm{feasible}, i}$, ${{J^\mathrm{UB}}}^\star \gets {J^\mathrm{UB}}^i$ and ${{J^\mathrm{LB}}}^\star \gets {J^\mathrm{LB}}^i$;

\end{algorithmic}
\end{algorithm}

Let $i$ denote the current iteration of the algorithm
and $I$ the maximum number of iterations.
Let $K^0$ denote the initial number of representative time periods (or clusters).
The derived upper and lower bounds are represented by $J^\mathrm{UB}$ and $J^\mathrm{LB}$, respectively.
At each iteration, the number of representative time periods is increased by a step size parameter $\rho$, and the bounds, along with a new feasible solution $\boldsymbol{z}^\mathrm{feasible}$ for the full-scale model, are recomputed.
The algorithm terminates when the optimality gap $\epsilon$, defined as
\begin{equation}\label{optimality_gap}
    \epsilon := \frac{J^\mathrm{UB} - J^\mathrm{LB}}{J^\mathrm{UB}},
\end{equation}
falls below a predefined threshold $\epsilon^\mathrm{thr}$.

Notably, the proposed algorithm not only provides valid objective function bounds regardless of the clustering technique employed for TSA,
but also generates a feasible solution to the full-scale models at each iteration.

\subsection{A Comparison with Benders Decomposition}
\label{subsec:comparison_benders_decomposition}

The proposed Algorithm~\ref{alg:TSA_algorithm}, detailed in the previous subsection, shares key similarities with the popular Benders decomposition method \cite{constante2025relaxation}.
As previously discussed, the full-scale models under consideration are two-stage problems, where the investment variables (referred to as \textit{complicating variables} in the context of decomposition methods) link the first and second stages.
When applied to these models, Benders decomposition separates the GEP problem into a \textit{master problem},
which determines the complicating variables and provides a lower bound on the optimal objective function value,
and a \textit{subproblem},
which optimizes the second-stage variables for fixed investment decisions, providing an upper bound on the optimal objective function value.
In the proposed Algorithm~\ref{alg:TSA_algorithm}, the aggregated models introduced in Subsection~\ref{subsec:aggregated_models} play the role of the Benders master problem, and the convex models \eqref{LP_MILP_full_scale_investment_model} and \eqref{LP_MIQP_full_scale_investment_model} correspond to the Benders subproblem.

While the Benders subproblem and the convex models in Algorithm~\ref{alg:TSA_algorithm} are analogous, the methods differ in how the investment variables are determined.
In Benders decomposition, the master problem includes only first-stage variables and constraints, and convergence is driven by Benders cuts generated from the subproblem, which accumulate over the iterations and progressively increase the master problem size.
In contrast, Algorithm~\ref{alg:TSA_algorithm} retains both first- and second-stage information, reducing dimensionality through variable aggregation.
Ultimately, the computational efficiency of Benders decomposition depends on the number of cuts generated, while that of Algorithm~\ref{alg:TSA_algorithm} depends on the number of representative periods (or clusters) employed for TSA.

A further distinction between Benders decomposition and the proposed TSA-based Algorithm~\ref{alg:TSA_algorithm} concerns the properties of the lower bounds.
Benders decomposition guarantees monotonically non-decreasing lower bounds \cite{constante2025relaxation}, a desirable property in optimization.
In contrast, the bounds in Algorithm~\ref{alg:TSA_algorithm} depend on the clustering technique employed.
Techniques such as k-means, which incorporate randomness, do not ensure that increasing the number of representative periods improves the aggregated model representativeness.
To provide a performance guarantee independent of the clustering technique, Algorithm~\ref{alg:TSA_algorithm} includes a conditional check to verify the lower bound improvement at each iteration.
Alternatively, a recursive clustering scheme that retains and refines partitions as the number of clusters increases can enforce monotonicity, yielding bounds akin to those of Benders decomposition.

Finally, a key potential advantage of the proposed Algorithm~\ref{alg:TSA_algorithm} compared to Benders decomposition resides in its greater flexibility.
Specifically, while the convergence of Benders decomposition depends on the quality of the Benders cuts, which are derived internally,
the convergence rate of Algorithm~\ref{alg:TSA_algorithm} is directly influenced by the representativeness of the aggregated model, which, in turn, depends on the selection of the representative periods.
As the decision-maker controls this selection, Algorithm~\ref{alg:TSA_algorithm} offers greater flexibility, enabling direct manipulation of the convergence rate.
Clearly, an accurate selection of the representative periods is contingent upon the inherent characteristics of the model, necessitating a thorough analysis of both the model's structure and the input time series in each specific case study.

\subsection{What Are We Clustering For?}
\label{subsec:what_are_we_clustering_for}
While the bounds derived in Subsections \ref{subsec:properties_aggregated_models} and \ref{subsec:performance_guaranteed_TSA} allow for approximating the optimal objective function value of the GEP model,
they do not directly inform the modeler about specific cost components (e.g., investment or operational costs) or individual decisions (e.g., technology-specific investments).
This raises the question: \textit{what are we clustering for?}—that is, which specific metric is the modeler seeking to approximate when performing TSA?
If the goal is to approximate the total cost, Algorithm~\ref{alg:TSA_algorithm} is effective.
However, when the focus shifts to stakeholder-specific metrics (e.g., the minimum wind capacity investment that preserves the global optimality gap of Algorithm~\ref{alg:TSA_algorithm}), additional steps are needed to interpret and refine the objective function bounds.
This subsection discusses strategies to address this need.

As a first observation, the interpretation of the bounds in Algorithm~\ref{alg:TSA_algorithm} is directly linked to the temporal scope of the GEP model.
When the horizon is limited, the investment costs dominate the objective function.
Conversely, as the horizon extends, the bounds primarily reflect the expected operational cost of the system,
as formalized in the following theoretical result.

\begin{proposition}\label{prop:long_term_operational_bounds}
Consider the infinite horizon GEP problem,
defined in the limit as the time horizon grows unbounded, i.e., $|\boldsymbol{T}| \to \infty$,
and formulated as the following full-scale MILP and MIQP models:

\noindent
\begin{minipage}{0.5\textwidth}
\upshape
\begin{subequations}\label{MILP_full_scale_investment_model_inf_horizon}
    \begin{align}
    \min_{\boldsymbol{z}} \quad & \lim_{|\boldsymbol{T}| \to \infty} \frac{1}{|\boldsymbol{T}|} \, J(\boldsymbol{z}) \label{MILP_full_scale_investment_model_inf_horizon_obj}\\
    \textrm{s.t.} \quad & \eqref{MILP_full_scale_investment_model_power_balance}-\eqref{MILP_full_scale_investment_model_binary}.
    \end{align}
    \end{subequations}
\end{minipage}
\hfill
\begin{minipage}{0.5\textwidth}
    \upshape
\begin{subequations}\label{MIQP_full_scale_investment_model_inf_horizon}
    \begin{align}
    \min_{\boldsymbol{z}} \quad & \lim_{|\boldsymbol{T}| \to \infty} \frac{1}{|\boldsymbol{T}|} \, \bar{J}(\boldsymbol{z}) \label{MIQP_full_scale_investment_model_inf_horizon_obj}\\
    \textrm{s.t.} \quad & \eqref{MILP_full_scale_investment_model_power_balance}-\eqref{MILP_full_scale_investment_model_binary}.
    \end{align}
    \end{subequations}
\end{minipage}

Let $J^{\mathrm{LB\infty}}$ and $J^{\mathrm{UB\infty}}$ denote the lower and upper bounds, respectively, obtained via Algorithm~\ref{alg:TSA_algorithm} for the MILP model~\eqref{MILP_full_scale_investment_model_inf_horizon},
and let $\bar{J}^{\,\mathrm{LB\infty}}$ and $\bar{J}^{\,\mathrm{UB\infty}}$ denote the corresponding bounds for the MIQP model~\eqref{MIQP_full_scale_investment_model_inf_horizon}.

Then, the following inequalities hold:
\begin{equation}\label{MILP_inf_bounds}
    J^{\mathrm{LB\infty}} \leq \mathbb{E}\left[ \boldsymbol{C^\mathrm{p}}^\top \boldsymbol{p}_{t} \, \Delta + C^\mathrm{ns} d^\mathrm{ns}_t \right] \leq J^{\mathrm{UB\infty}},
\end{equation}
\begin{equation}\label{MIQP_inf_bounds}
    \bar{J}^{\,\mathrm{LB\infty}} \leq \mathbb{E}\left[ \boldsymbol{C^\mathrm{p}}^\top \boldsymbol{p}_{t} \, \Delta + C^\mathrm{ns} d^\mathrm{ns}_t + \left\| \mathbf{A} \boldsymbol{z}^{\boldsymbol{\mathrm{op}}}_t - \boldsymbol{Z}^{\boldsymbol{\mathrm{ref}}}_t \right\|_2^2 \right] \leq \bar{J}^{\,\mathrm{UB\infty}},
\end{equation}
where $\mathbb{E}[\cdot]$ denotes the expected value.
\end{proposition}

\begin{proof}
We first demonstrate the result in~\eqref{MILP_inf_bounds} for the MILP model \eqref{MILP_full_scale_investment_model_inf_horizon}.

By expanding the objective function  \eqref{MILP_full_scale_investment_model_inf_horizon_obj}, we obtain:
\begin{multline*}
    \lim_{|\boldsymbol{T}| \to \infty} \frac{1}{|\boldsymbol{T}|} \left( \boldsymbol{C^\mathrm{inv}}^\top \boldsymbol{x} + \sum_{t \in \boldsymbol{T}} \left( \boldsymbol{C^\mathrm{p}}^\top \boldsymbol{p}_{t} \, \Delta + C^\mathrm{ns} \, d^\mathrm{ns}_t\right) \right)\\
    = \lim_{|\boldsymbol{T}| \to \infty} \frac{1}{|\boldsymbol{T}|} \sum_{t \in \boldsymbol{T}} \left( \boldsymbol{C^\mathrm{p}}^\top \boldsymbol{p}_{t} \, \Delta + C^\mathrm{ns} \, d^\mathrm{ns}_t\right) = \mathbb{E}\left[ \boldsymbol{C^\mathrm{p}}^\top \boldsymbol{p}_{t} \, \Delta + C^\mathrm{ns} \, d^\mathrm{ns}_t \right],
\end{multline*}
where the last equality follows from the strong law of large numbers.

Similarly, for the MIQP model \eqref{MIQP_full_scale_investment_model_inf_horizon}, we obtain:
\begin{multline*}
    \lim_{|\boldsymbol{T}| \to \infty} \frac{1}{|\boldsymbol{T}|} \left(J(\boldsymbol{z}) + \left\| \mathbf{A} \boldsymbol{z}^{\boldsymbol{\mathrm{op}}}_t - \boldsymbol{Z}^{\boldsymbol{\mathrm{ref}}}_t \right\|_2^2\right)\\
    = \mathbb{E}\left[ \boldsymbol{C^\mathrm{p}}^\top \boldsymbol{p}_{t} \, \Delta + C^\mathrm{ns} \, d^\mathrm{ns}_t \right] + \mathbb{E}\left[ \left\| \mathbf{A} \boldsymbol{z}^{\boldsymbol{\mathrm{op}}}_t - \boldsymbol{Z}^{\boldsymbol{\mathrm{ref}}}_t \right\|_2^2 \right],
\end{multline*}
from which the result in \eqref{MIQP_inf_bounds} follows directly.
\end{proof}

In essence, Proposition~\ref{prop:long_term_operational_bounds} provides a practical interpretation of the objective function bounds derived via Algorithm~\ref{alg:TSA_algorithm}, 
indicating that they could be employed to closely approximate the total investment costs in short-term GEP, and the total operational costs in long-term GEP.

While Proposition~\ref{prop:long_term_operational_bounds} facilitates interpreting the objective function bounds of Algorithm~\ref{alg:TSA_algorithm} in terms of the GEP model’s temporal scope,
the challenge persists in tightening these bounds for alternative stakeholder-specific metrics beyond the total investment and operational costs.
To rigorously derive error bounds for any stakeholder-specific metric, an auxiliary optimization step can be introduced at the final iteration of Algorithm~\ref{alg:TSA_algorithm}.
Let $M(\boldsymbol{z})$ denote the stakeholder-specific metric. Then, upper and lower bounds on $M(\boldsymbol{z})$, consistent with the global optimality gap achieved by Algorithm~\ref{alg:TSA_algorithm}, can be computed by solving the metric-specific models

\noindent
\begin{minipage}{0.49\textwidth}
    \begin{subequations}\label{metric_max_MILP}
    \begin{align}
    \max_{\boldsymbol{z}} \quad & M(\boldsymbol{z})\\
    \textrm{s.t.} \quad & \eqref{MILP_full_scale_investment_model_power_balance}-\eqref{MILP_full_scale_investment_model_binary}, \nonumber \\
    & J^\mathrm{LB} \leq J(\boldsymbol{z}) \leq J^\mathrm{UB}, \label{metric_max_MILP_bounds}
    \end{align}
    \end{subequations}
\end{minipage}
\hfill
\begin{minipage}{0.49\textwidth}
    \begin{subequations}\label{metric_min_MILP}
    \begin{align}
    \min_{\boldsymbol{z}} \quad & M(\boldsymbol{z})\\
    \textrm{s.t.} \quad & \eqref{MILP_full_scale_investment_model_power_balance}-\eqref{MILP_full_scale_investment_model_binary}, \nonumber \\
    & J^\mathrm{LB} \leq J(\boldsymbol{z}) \leq J^\mathrm{UB}, \label{metric_min_MILP_bounds}
    \end{align}
    \end{subequations}
\end{minipage}
when considering the MILP model \eqref{MILP_full_scale_investment_model}, or the metric-specific models

\noindent
\begin{minipage}{0.49\textwidth}
    \begin{subequations}\label{metric_max_MIQP}
    \begin{align}
    \max_{\boldsymbol{z}} \quad & M(\boldsymbol{z})\\
    \textrm{s.t.} \quad & \eqref{MILP_full_scale_investment_model_power_balance}-\eqref{MILP_full_scale_investment_model_binary}, \nonumber \\
    & \bar{J}^{\,\mathrm{LB}} \leq \bar{J}(\boldsymbol{z}) \leq \bar{J}^{\,\mathrm{UB}}, \label{metric_max_MIQP_bounds}
    \end{align}
    \end{subequations}
\end{minipage}
\hfill
\begin{minipage}{0.49\textwidth}
    \begin{subequations}\label{metric_min_MIQP}
    \begin{align}
    \min_{\boldsymbol{z}} \quad & M(\boldsymbol{z})\\
    \textrm{s.t.} \quad & \eqref{MILP_full_scale_investment_model_power_balance}-\eqref{MILP_full_scale_investment_model_binary}, \nonumber \\
    & \bar{J}^{\,\mathrm{LB}} \leq \bar{J}(\boldsymbol{z}) \leq \bar{J}^{\,\mathrm{UB}}, \label{metric_min_MIQP_bounds}
    \end{align}
    \end{subequations}
\end{minipage}
when considering the MIQP model \eqref{MIQP_full_scale_investment_model}.

In words, solving the optimization models \eqref{metric_max_MILP}–\eqref{metric_min_MILP} for the MILP formulation, and \eqref{metric_max_MIQP}–\eqref{metric_min_MIQP} for the MIQP formulation of the GEP,
in combination with the objective function bounds of Algorithm~\ref{alg:TSA_algorithm},
yields the range of values for the metric $M(\boldsymbol{z})$ that preserve the desired global optimality gap.

Moreover, solving \eqref{metric_max_MILP}--\eqref{metric_min_MILP} or \eqref{metric_max_MIQP}--\eqref{metric_min_MIQP} in conjunction with Algorithm~\ref{alg:TSA_algorithm} offers various advantages over directly addressing the full-scale models.
First, the feasible solution derived via Algorithm~\ref{alg:TSA_algorithm} can be used to warm-start \eqref{metric_max_MILP}–\eqref{metric_min_MILP} and \eqref{metric_max_MIQP}–\eqref{metric_min_MIQP}.
More importantly, compared to the original full-scale models,
the metric-specific models incorporate additional constraints, namely \eqref{metric_max_MILP_bounds}, \eqref{metric_min_MILP_bounds}, \eqref{metric_max_MIQP_bounds}, and \eqref{metric_min_MIQP_bounds}, that restrict the feasible region to solutions consistent with the objective function bounds.
This tightening of the feasible region not only improves computational efficiency compared to solving the full-scale models directly,
but also provides targeted insight into the metric of interest, while retaining the performance guarantee of Algorithm~\ref{alg:TSA_algorithm}.

Notably, if \eqref{metric_min_MILP} and \eqref{metric_min_MIQP} remain computationally demanding,
an LP relaxation may be applied to obtain a valid lower bound on the metric.
Similarly, the metric-specific models \eqref{metric_max_MILP} and \eqref{metric_max_MIQP} can be simplified by fixing variables,
e.g., to the values derived from the feasible full-scale solution obtained via Algorithm~\ref{alg:TSA_algorithm},
to establish a valid upper bound on the metric.

\section{Numerical Results}
\label{sec:numerical_results}
This section reports the numerical results.
Subsection~\ref{subsec:obj_fun_error_bounds_results} presents the objective function bounds of Algorithm~\ref{alg:TSA_algorithm}.
Subsection~\ref{subsec:Benders_results} compares the proposed algorithm with classical Benders decomposition.
Subsection~\ref{subsec:what_are_we_clustering_for_results} illustrates how the objective function bounds inform stakeholder-specific metrics.

The input energy demand and capacity factor time series, as well as the simulation parameters used in the following, are sourced from \cite{santosuosso2025optimal}.
The non-supplied energy cost is 5000 \euro/MWh,
with investment costs of 40000 \euro/MW for thermal generators, 30000 \euro/MW for vRES, and 35000 \euro/MW for storage units. 
The operational costs are 50 \euro/MWh for thermal and 3 \euro/MWh for vRES.
The generator and storage investment capacities are constrained between 0.1 and 1 MW, with 20\% thermal generators and 80\% vRES units.
The capacity factor for thermal generators is fixed at $F_t = 1$ for all time periods.
The maximum storage charging and discharging power limits are generated as $\boldsymbol{\overline{P}}^{\boldsymbol{\mathrm{s}}} \sim \mathcal{U}(0.3, 0.6)$ MW,
where $\mathcal{U}(a,b)$ denotes a uniform distribution on the interval $[a,b]$,
while $\boldsymbol{\underline{P}}^{\boldsymbol{\mathrm{s}}} = \mathbf{0}_{2 N}$.
The charging and discharging efficiencies are set to 0.9 and 1.1, respectively, with initial storage states set to 0 MWh.
In the MIQP formulation, we take $\mathbf{A} = \begin{bmatrix} \mathbf{0}_{N \times G + 1}, \mathbf{I}_{N}, \mathbf{0}_{N \times 2 N} \end{bmatrix}$,
enforcing reference values for the storage states, sampled from $\mathcal{U}(0.1, 1)$.
We set $\Delta = 1$ hour, $I = 1000$, and $\epsilon^\mathrm{thr} = 0.01$ (1\% error).
The simulations are conducted on an Intel i7 processor with 32 GB of RAM, using Gurobi 12.0.1.

\subsection{The Objective Function Bounds}
\label{subsec:obj_fun_error_bounds_results}

\begin{figure}[t]
\centering
\includegraphics[width=1\textwidth]{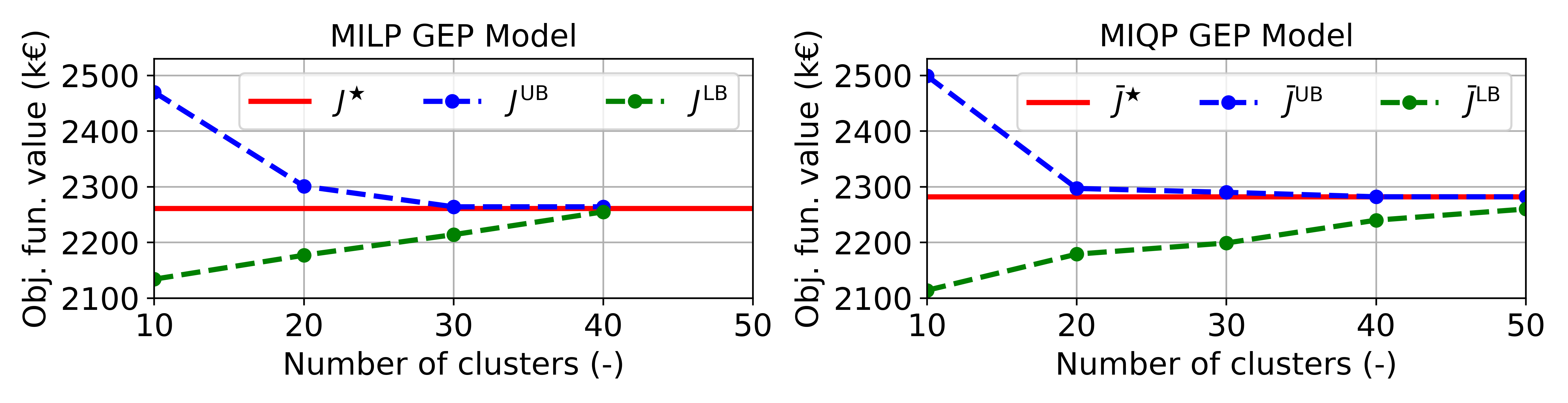}
\caption{Illustrative examples of objective function bounds derived using the proposed Algorithm~\ref{alg:TSA_algorithm} for the MILP and MIQP formulations of the GEP problem.}\label{fig1}
\end{figure}

A key advantage of the proposed Algorithm~\ref{alg:TSA_algorithm} over several existing TSA-based solution methods for GEP is its ability to derive objective function bounds, offering a clear performance guarantee.
Figure~\ref{fig1} illustrates the obtained bounds for a stylized example with $G=10$, $N=10$, and $T=500$ hours,
using sequential clustering (randomly splitting the time periods while maintaining temporal chronology), $K^0 = 10$, and $\rho = 10$.
In both MILP and MIQP cases, the model size is reduced by one order of magnitude (40 or 50 clusters instead of 500 periods).
Notably, the upper bound, derived as a projection of the aggregated solution into the original feasible space of the GEP problem, 
remains of high quality in early iterations, while the lower bound progressively refines.
For the MILP model, the sequential clustering typically results in a linear evolution of the lower bound.
This linearity, however, is not generally retained in the MIQP model, as shown in the figure.

\begin{figure}[t]
\centering
\includegraphics[width=1\textwidth]{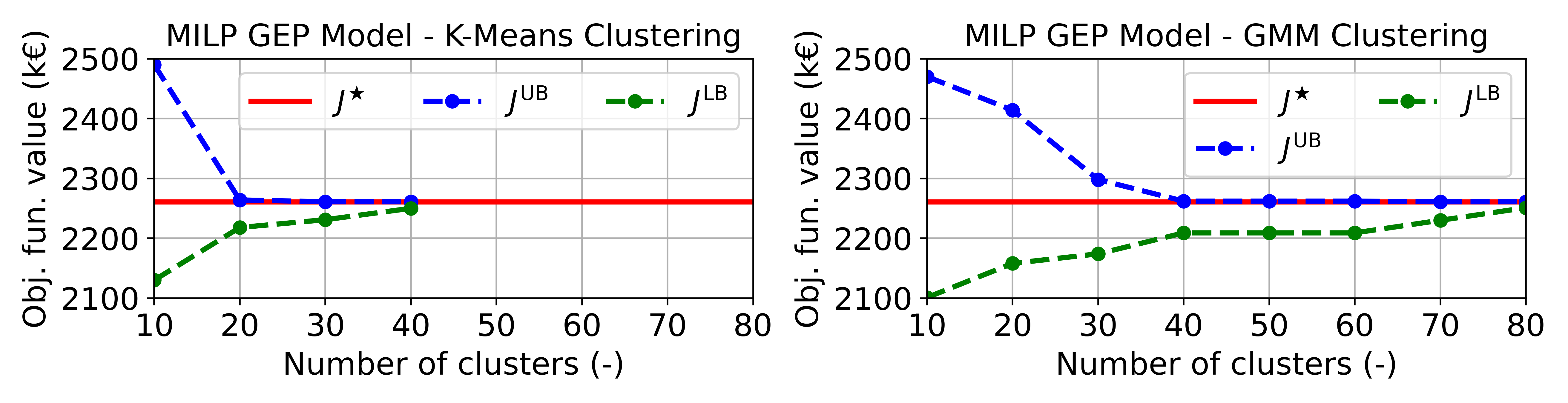}
\caption{Illustrative examples of objective function bounds derived using the proposed Algorithm~\ref{alg:TSA_algorithm} for the MILP GEP model, with different clustering techniques.}\label{fig2}
\end{figure}

Although the performance guarantee of Algorithm~\ref{alg:TSA_algorithm} is independent of the clustering technique employed,
the quality of the resulting bounds depends on this choice.
Figure~\ref{fig2} compares the bounds obtained using k-means and GMM in the aforementioned stylized example.
As neither technique preserves temporal chronology, a postprocessing step is applied to recover the closest temporally ordered set of clusters. 
K-means yields convergence within four iterations, while GMM requires twice as many due to its (unmet) assumption of Gaussian-distributed input data,
resulting in increasingly complex aggregated models with minimal improvement in the approximation quality.

\begin{table}[t]
\centering
\begin{tabular}{|c|c|c|c|c|}
\hline
\multirow{3}{*}{\textbf{Settings}} & \multicolumn{4}{|c|}{\textbf{Runtime (min)}}\\
\cline{2-5}
& \multicolumn{2}{|c|}{\textbf{MILP GEP model}} & \multicolumn{2}{|c|}{\textbf{MIQP GEP model}}\\
\cline{2-5}
& \textbf{Full model} & \textbf{Alg.~\ref{alg:TSA_algorithm}} & \textbf{Full model} & \textbf{Alg.~\ref{alg:TSA_algorithm}}\\
\hline
\multicolumn{1}{|l|}{$T = 8760$} & \multirow{2}{*}{$161.8$} & $174.4$ & \multirow{2}{*}{$496.1$} & $388.1$\\
\multicolumn{1}{|l|}{$N = G = 1500$} & & $(\boldsymbol{+ 7.8 \%})$ & & $(\boldsymbol{- 21.8} \%)$\\
\hline
\multicolumn{1}{|l|}{$T = 8760$} & \multirow{2}{*}{$946.6$} & $833.1$ & \multirow{2}{*}{$\infty$} & $1069.8$\\
\multicolumn{1}{|l|}{$N = G = 3000$} & & $(\boldsymbol{- 12.0 \%})$ & & $(\boldsymbol{- \infty} \%)$\\
\hline
\multicolumn{1}{|l|}{$T = 43800$} & \multirow{2}{*}{$1334.7$} & $1016.4$ & \multirow{2}{*}{$\infty$} & $1426.6$\\
\multicolumn{1}{|l|}{$N = G = 3000$} & & $(\boldsymbol{- 23.8 \%})$ & & $(\boldsymbol{- \infty} \%)$\\
\hline
\end{tabular}
\caption{Runtimes of Algorithm~\ref{alg:TSA_algorithm} versus full-scale optimization, with relative differences in brackets. The symbol $\infty$ denotes intractable instances of the GEP problem.}\label{tab1}
\end{table}

Table~\ref{tab1} presents the runtimes of Algorithm~\ref{alg:TSA_algorithm} and full-scale optimization, with $K^0 = T / 10$, $\rho = 100$, and sequential clustering.
The models that could not be solved within 24 hours are denoted by $\infty$.
In the MILP case, Algorithm~\ref{alg:TSA_algorithm} shows no computational advantage for small-scale problems.
However, as the problem size increases,
particularly when extending $T$ from 8760 hours (1 year) to 43800 hours (5 years),
a significant advantage is observed relative to full-scale optimization.
In the MIQP case, Algorithm~\ref{alg:TSA_algorithm} reduces the runtime by 20\% for small-scale problems
and, more importantly, restores tractability, whereas full-scale optimization fails.

\subsection{Results from the Comparison with Benders Decomposition}
\label{subsec:Benders_results}

As discussed in Subsection~\ref{subsec:comparison_benders_decomposition}, the proposed TSA-based Algorithm~\ref{alg:TSA_algorithm} offers key potential advantages over classical Benders decomposition by adapting to the specific structure of the GEP model.
Figure~\ref{fig3} illustrates this by comparing the optimality gaps achieved by both methods for a MILP model with $G = 25$, $N = 25$, and $T = 500$ hours.
While classical Benders reduces the gap to 5\% within a few iterations, 50 additional iterations are required to reach the target 1\%,
due to the lack of direct control over cut selection.
In contrast, the proposed algorithm allows direct clustering technique selection and parameters tuning, improving flexibility and reducing the iterations required for convergence by 44 compared to classical Benders decomposition.

\begin{figure}[t]
\centering
\includegraphics[width=1\textwidth]{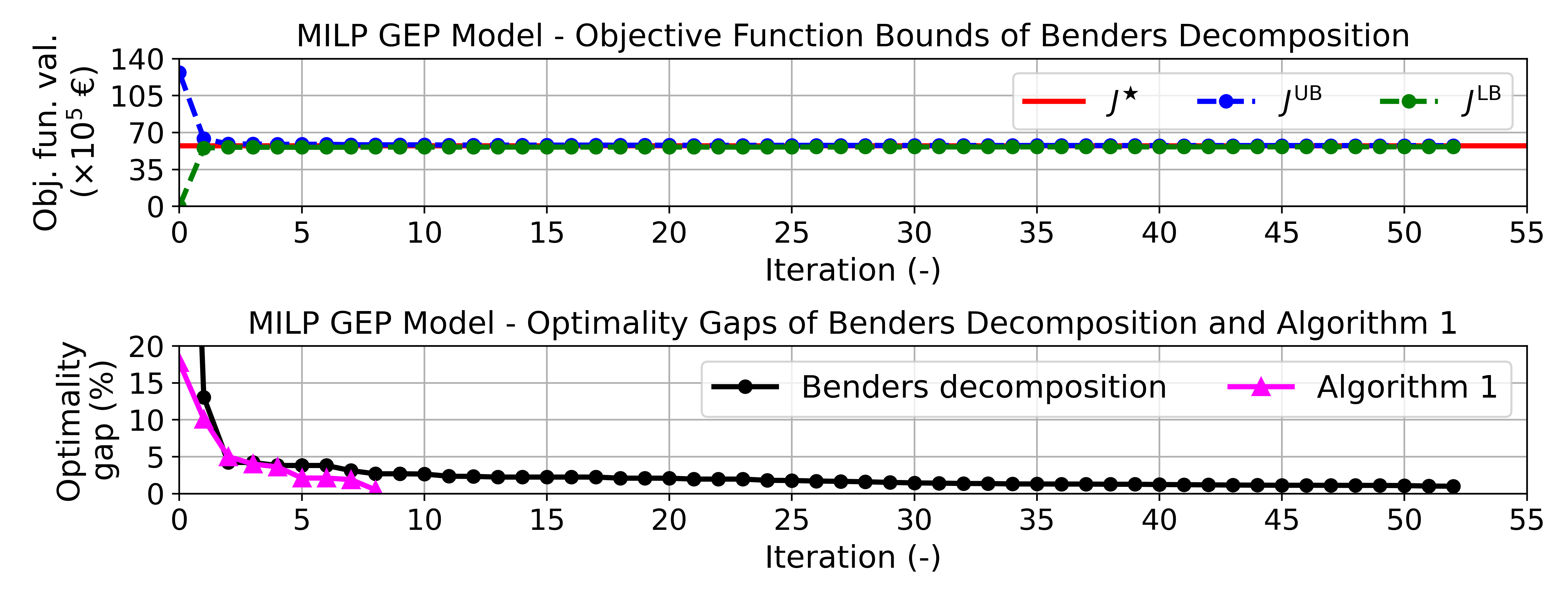}
\caption{Illustrative example of objective function bounds and optimality gaps derived from Benders decomposition and the proposed TSA-based Algorithm~\ref{alg:TSA_algorithm}.}\label{fig3}
\end{figure}

\begin{figure}[t]
\centering
\includegraphics[width=1\textwidth]{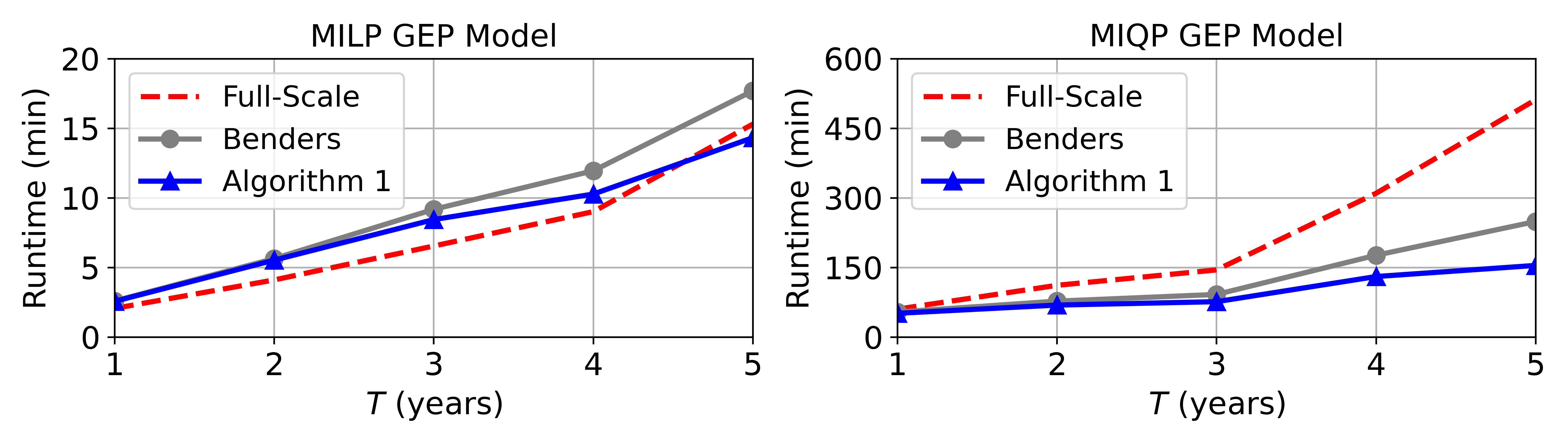}
\caption{Runtimes of Algorithm~\ref{alg:TSA_algorithm}, full-scale optimization, and Benders decomposition when applied to MILP and MIQP GEP models with increasing time horizon $T$.}\label{fig4}
\end{figure}

The enhanced flexibility of Algorithm~\ref{alg:TSA_algorithm} yields not only fewer iterations but also significant runtime reductions,
as shown in Figure~\ref{fig4} for both MILP and MIQP models with $G = 50$ and $N = 50$.
In the best case ($T = 5$ years), using $K^0 = T / 100$, $\rho = 100$, and sequential clustering,
the algorithm reduces runtime by 7\% and 19\% compared to full-scale optimization and Benders, respectively, for the MILP model, and by 70\% and 38\% for the MIQP model.

\subsection{Results for the Stakeholder-Specific Metrics}
\label{subsec:what_are_we_clustering_for_results}

\begin{figure}[t]
\centering
\includegraphics[width=1\textwidth]{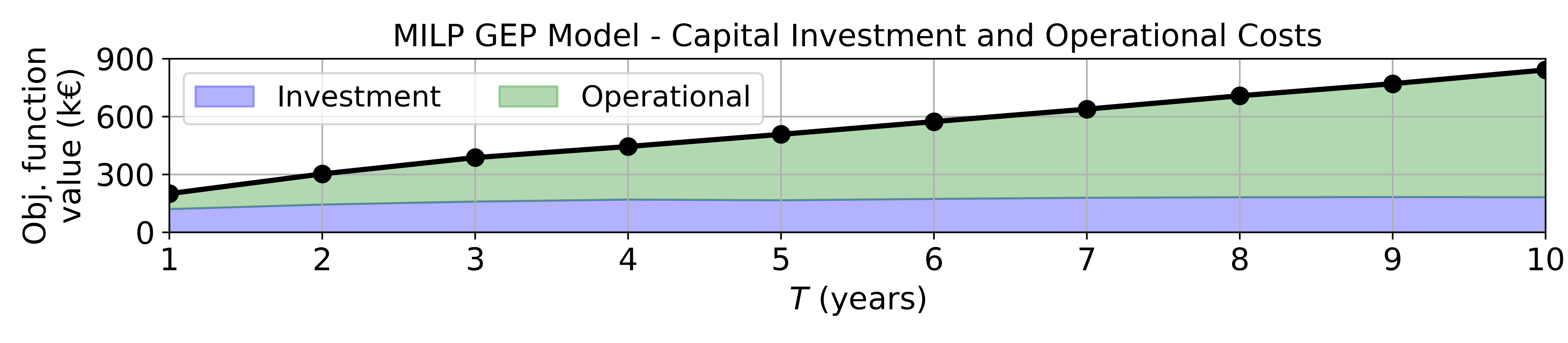}
\caption{Capital investment and operational cost terms within the optimal objective function value of a MILP GEP model as the time horizon $T$ increases.}\label{fig5}
\end{figure}

\begin{figure}[t]
\centering
\includegraphics[width=1\textwidth]{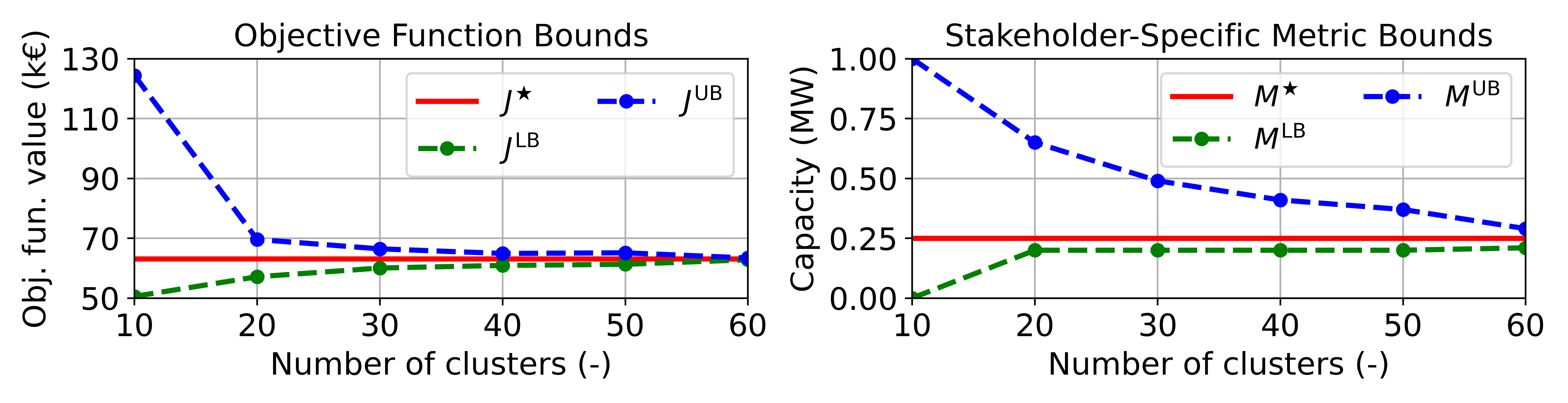}
\caption{Objective function bounds and corresponding bounds for the stakeholder-specific metric of interest, namely the capacity investment in a specific storage unit.}\label{fig6}
\end{figure}

As discussed in Subsection~\ref{subsec:what_are_we_clustering_for}, existing performance-guaranteed TSA methods provide bounds only on the GEP objective, which may not reflect stakeholder-specific metrics.
These bounds should therefore be interpreted and adapted in relation to the specific metric of interest.
As demonstrated in Proposition~\ref{prop:long_term_operational_bounds} and shown in Figure~\ref{fig5}, capital investment costs dominate the objective function in short-term GEP, while operational costs become more prominent as $T$ increases.
Accordingly, the bounds from Algorithm~\ref{alg:TSA_algorithm} serve as
proxies for investment or expected operational costs, depending on $T$.

Consider a stakeholder using the MILP GEP model to analyze the development of a power system with a focus on a metric $M$,
such as the optimal capacity investment in a specific storage unit.
The objective function bounds from Algorithm~\ref{alg:TSA_algorithm}, combined with the metric-specific models in Subsection~\ref{subsec:what_are_we_clustering_for}, provide upper ($M^\mathrm{UB}$) and lower ($M^\mathrm{LB}$) bounds on the optimal metric value $M^\star$.
Figure~\ref{fig6} illustrates these bounds, offering iteratively refined capacity ranges for the storage unit of interest.
Notably, the metric-specific lower bound $M^\mathrm{LB}$ holds a clear practical significance: it represents the minimum (no-regret) capacity investment needed to preserve the global optimality gap achieved by Algorithm~\ref{alg:TSA_algorithm}.
For instance, Figure~\ref{fig6} shows that, at early iterations, the stakeholder is informed that a minimum investment of $\sim0.25$ MW in the storage unit is required to ensure optimal GEP decisions.

\section{Conclusion and Future Work}
\label{sec:conclusion}
This study addresses the GEP problem under both MILP and MIQP formulations.
As the problem size increases, nonconvex dynamics, heterogeneous asset interactions, and time-coupling constraints often render the GEP model computationally intensive or intractable.
To mitigate this, we propose a TSA-based solution method with rigorous performance guarantees, providing theoretically validated bounds on the maximum objective function error relative to the full-scale model.
Our iterative algorithm consistently produces valid bounds regardless of the clustering technique used and yields a feasible solution to the full-scale model at each iteration.
We provide theoretical and numerical comparisons with the Benders decomposition method.
Moreover, unlike existing approaches focused solely on objective function value approximations, our method extends the bounds to stakeholder-specific metrics.

Numerical results confirm the validity and tightness of the derived bounds, showing significant computational gains over Benders decomposition and full-scale optimization.
Notably, our method restores tractability for large-scale MIQP models.
An illustrative example further highlights how the objective function bounds can be translated into stakeholder-specific metric bounds.

Future work will investigate heuristics for clustering selection informed by the model structure and adaptive parameter tuning. The methodology will also be validated on alternative capacity expansion problems beyond GEP.

\section*{Acknowledgment}
Funded by the European Union (ERC, NetZero-Opt, 101116212). Views and opinions expressed are however those of the authors only and do not necessarily reflect those of the European Union or the European Research Council. Neither the European Union nor the granting authority can be held responsible for them.

\end{document}